\documentclass{article}
\usepackage{times}
\usepackage{epsfig}
\usepackage{graphicx}
\usepackage{graphics}
\usepackage{amsmath,amsthm}
\usepackage{amssymb}
\usepackage{dsfont}
\usepackage{xspace}
\usepackage{color}
\usepackage{enumitem}
\usepackage{thmtools}
\usepackage{tikz}
\usepackage{lscape}
\usepackage{lscape}
\usepackage{subfigure}
\usepackage{multirow} 
\usepackage{wrapfig}
\usepackage{algorithm2e}
\usepackage{mathtools}
\usepackage{MnSymbol}
\usepackage{mathrsfs}
\usepackage{authblk}

\usepackage[left=3.9cm,right=3.9cm,top=3cm,bottom=3cm]{geometry}

\DeclareMathOperator*{\argmin}{argmin}
\numberwithin{equation}{section}

\definecolor{gnuplotbrown}{RGB}{165,42,42}
\definecolor{gnuplotblue}{RGB}{0,0,255}
\definecolor{gnuplotred}{RGB}{255,0,0}

\newcommand{\R}{{\mathds{R}}}

\newcommand{\differential}{{\,\mathrm{d}}}
\newcommand{\D}{{\mathrm{D}}}

\renewcommand{\div}{\operatorname{div}}

\newcommand{\inputImage}{u_{0}} \newcommand{\image}{u} \newcommand{\IB}{I_{\overline{B}_1}}

\newcommand{\Id}{\mathrm{Id}}
\newcommand{\err}{\mathrm{err}}
\newcommand{\jumpArea}{\mathrm{a}}

\newcommand{\Hausdorff}{\mathcal{H}}
\newcommand{\Lebesgue}{\mathcal{L}}
\newcommand{\EnerMS}{E}
\newcommand{\ERel}{E^\text{rel}}
\newcommand{\DRel}{{D^\text{rel}}}

\newcommand{\EnerS}{{\ERel_s}}

\newcommand{\discrete}[1]{\mathbf{#1}_h}
\newcommand{\quadMesh}{\mathcal{M}_h}
\newcommand{\simplexMesh}{\mathcal{S}_h}

\newcommand{\V}{\mathcal{V}}
\newcommand{\Q}{\mathcal{Q}}
\newcommand{\ff}{{\theta}}
\newcommand{\Ph}{\discrete{P}}
\newcommand{\Qh}{\discrete{Q}}
\newcommand{\Uh}{\discrete{U}}
\newcommand{\barUh}{\bar{\mathbf{U}}_h}
\newcommand{\Vh}{\discrete{V}}
\newcommand{\Wh}{\discrete{W}}

\newcommand{\proj}{\mathcal{P}_h}

\newcommand{\ffhone}{{\Theta_{1,h}}}
\newcommand{\ffhtwo}{{\Theta_{2,h}}}
\newcommand{\ffhonei}{{\Theta^i_{1,h}}}
\newcommand{\ffhtwoi}{{\Theta^i_{2,h}}}

\newcommand{\Mass}{{\discrete{M}}}
\newcommand{\Stiff}{{\discrete{S}}}
\newcommand{\tildeMass}{\mathbf{\tilde M}_h}
\newcommand{\Fh}{{\mathbf{F}_h}}
\newcommand{\Fhstar}{{\mathbf{F}^\ast_h}}
\newcommand{\Gh}{{\mathbf{G}_h}}
\newcommand{\Ghstar}{{\mathbf{G}^\ast_h}}
\newcommand{\Lh}{{\Lambda_h}}
\newcommand{\Lhstar}{{\Lambda^\ast_h}}
\newcommand{\ERelh}{{\mathbf{E}_h^{rel}}}
\newcommand{\DRelh}{{\mathbf{D}_h^{rel}}}

\newcommand{\neigh}{{\mathscr{N}}}

\newcommand{\beq}{\begin{equation*}}
\newcommand{\eeq}{\end{equation*}}
\newcommand{\beqn}{\begin{equation}}
\newcommand{\eeqn}{\end{equation}}
\newcommand{\beqa}{\begin{align*}}
\newcommand{\eeqa}{\end{align*}}
\newcommand{\beqan}{\begin{align}}
\newcommand{\eeqan}{\end{align}}

\theoremstyle{plain}\newtheorem{theorem}{Theorem}[section]

\newtheorem{proposition}[theorem]{Proposition}

\theoremstyle{definition}

\theoremstyle{remark}

\def\XXint#1#2#3{{\setbox0=\hbox{$#1{#2#3}{\int}$}
\vcenter{\hbox{$#2#3$}}\kern-.5\wd0}}

\makeatletter
\DeclareRobustCommand\onedot{\futurelet\@let@token\@onedot}
\def\@onedot{\ifx\@let@token.\else.\null\fi\xspace}

\def\eg{\emph{e.g}\onedot} 
\def\ie{\emph{i.e}\onedot}

\def\wrt{w.r.t\onedot} 
\def\etal{\emph{et al}\onedot}

\makeatother

\graphicspath{{images/}}

\definecolor{Red}{rgb}{0.8,0,0}
\definecolor{Blue}{rgb}{0,0,0.8}
\definecolor{Green}{rgb}{0,0.4,0.4}

\definecolor{darkblue} {rgb} {0.00, 0.0, 1.0}

\begin{document}
\renewcommand\Affilfont{\small}
\title{A Posteriori Error Control\\ for the Binary Mumford-Shah Model}
\author[1]{Benjamin Berkels}
\author[2]{Alexander Effland} 
\author[2]{Martin Rumpf}
\affil[1]{AICES Graduate School, RWTH Aachen University, \url{berkels@aices.rwth-aachen.de}}
\affil[2]{Institute for Numerical Simulation, University of Bonn,\authorcr $\{$\url{alexander.effland}, \url{martin.rumpf}$\}$\url{@ins.uni-bonn.de}}
\maketitle
\begin{abstract}
The binary Mumford-Shah model is a widespread tool for image segmentation and can be considered as 
a basic model in shape optimization with a broad range of applications in computer vision, ranging from basic segmentation and labeling
to object reconstruction. This paper presents robust a posteriori error estimates for a natural 
error quantity, namely the area of the non properly segmented region.
To this end, a suitable strictly convex and non-constrained relaxation of the originally non-convex functional is 
investigated and Repin's functional approach for a posteriori error estimation 
is used to control the numerical error for the relaxed problem in the $L^2$-norm. In combination with a suitable 
cut out argument, a fully practical estimate for the area mismatch is derived.
This estimate is incorporated in an adaptive meshing strategy.
Two different adaptive primal-dual finite element schemes, 
and the most frequently used finite difference discretization are investigated and compared.
Numerical experiments show qualitative and quantitative properties of the estimates 
and demonstrate their usefulness in practical applications.
\end{abstract}

\section{Introduction}\label{sec:intro}
Since the introduction of the image denoising and edge segmentation model by Mumford and Shah in the late 80's \cite{MuSh89},  there has 
been much effort to find effective and efficient numerical algorithms to compute minimizers of different variants of this variational problem. 
The original model is based on the functional 
$
E_\text{MS}[\image,K]=\int_{\Omega\backslash K}|\nabla\image|^2+\alpha(\image-\inputImage)^2\differential x+\beta\Hausdorff^{n-1}(K)
$
with $\alpha,\beta>0$, where $u_0:\Omega \to \R$ is a scalar image intensity on the image domain $\Omega\subset\R^n$, $u$ the reconstructed image intensity and $K$ the associated set of edges, on which the image intensity $u$ jumps.   
Here, $\Hausdorff^{n-1}$ denotes the $(n-1)$-dimensional Hausdorff measure. 
The space of functions of bounded variation $BV(\Omega)$ turned out to be the proper space to formulate the problem in a mathematical rigorous way. Indeed, existence in the context of the space of special functions of bounded variation $SBV(\Omega)$ was proved by Ambrosio 
(see \cite[Theorem 4.2]{Am89}). For details on these spaces, we refer to \cite{AmFuPa00}.
Restricting $u$ to be piecewise constant instead of piecewise smooth, one is lead to a basic and widespread image segmentation model. This model is discussed from a geometric perspective in the book by Morel and Solimini \cite{Morel1995}.
In the case of just two intensity values $c_1$ and $c_2$, the associated energy can be rewritten in terms of a characteristic function $\chi \in BV(\Omega,\{0,1\})$ as
\begin{equation}
\EnerMS[\chi,c_1,c_2] =\int_{\Omega}\ff_1\chi+\ff_2(1-\chi)\differential x +|\D\chi|(\Omega)\,.
\label{eq:OriginalBinaryModel}
\end{equation}
Here, $\ff_i = \frac{1}{\nu}(c_i-\inputImage)^2$ for $i=1,2$, the new weight $\nu = \beta/\alpha $ and the resulting binary model is given by $u= c_1 \chi + c_2 (1-\chi)$.
For fixed $\chi$, one immediately obtains the optimal constants $c_1 = (\int_\Omega \chi \differential x)^{-1} \int_\Omega \chi u_0 \differential x$ and $c_2 = (\int_\Omega 1-\chi \differential x)^{-1} \int_\Omega (1-\chi) u_0 \differential x$. For fixed $c_1$ and $c_2$ one aims at minimizing the energy over the non-convex set of characteristic functions $\chi \in BV(\Omega,\{0,1\})$.
Nikolova, Esedoglu and Chan \cite{NiEsCh06} showed that this non-convex minimization problem can be solved via relaxation and thresholding---a breakthrough for both reliable and fast algorithms in computer vision \cite{PoCrBi10, ChCrPo12}. Here, at first one asks for a minimizer of $\EnerMS[\cdot,c_1,c_2]$ over all $u \in BV(\Omega,[0,1])$ and then thresholds $u$ for any threshold value $s\in [0,1)$ to obtain the solution $\chi= \chi_{[u>s]}$ of the original minimization problem. 
The relaxed problem coincides with a constrained version of the classical image denoising model by Rudin, Osher and Fatemi (ROF) \cite{RuOsFa92}.
Numerical schemes for an effective and efficient minimization of this model have extensively been studied.
Making use of a dual formulation, Chambolle \cite{Ch04} introduced an iterative finite difference scheme and proved its convergence.
Hinterm\"uller and Kunisch \cite{HiKu04} proposed a predual formulation for a generalized ROF model and proposed to apply a semismooth Newton method for a regularized variant.
Chambolle and Pock \cite{ChPo11} deduced a primal-dual algorithm with guaranteed first order convergence and applied their approach to different variational models in $BV$ such as  
image denoising, deblurring and interpolation.
The scheme is based on an alternating discrete gradient scheme for the discrete primal and the discrete dual problem. 
Bartels \cite{Ba13} used the embedding $BV(\Omega) \cap L^\infty(\Omega) \hookrightarrow H^{\frac12}(\Omega)$ to improve the step-size restriction
for $BV$ functionals.
Wang and Lucier \cite{WaLu11} employed a finite difference approximation of the ROF model and derived 
an a priori error estimate for the discrete solution based on suitable projection operators.
Following Dobson and Vogel \cite{DoVo97}, the total variation regularization can be approximated smoothly via $\sqrt{|\nabla u|^2 +\epsilon}$.
In \cite{FePr02}, the convergence of the $L^2$-gradient flow of this smooth approximation to the TV flow in $L^2$ is shown
under strong regularity assumptions on the solution.

Furthermore, approximations of the original Mumford-Shah model have been studied extensively.
An early overview of different approximation and discretization strategies was given by Chambolle in \cite{Ch95}.
Ambrosio and Tortorelli \cite{AmTo92} proposed a phase field approximation 
of this functional and proved its $\Gamma$-convergence.
Chambolle and Dal Maso \cite{ChDa99} proposed a discrete finite element approximation and established its $\Gamma$-convergence.
Bourdin and Chambolle \cite{BoCh00} picked up this approach and studied the generation of adaptive  meshes 
iteratively adapted in accordance to an anisotropic metric depending on the current approximate solution.
In \cite{Sh05}, Shen introduced a $\Gamma$-converging approximation of the piecewise constant Mumford-Shah segmentation, where the length
term in the Mumford-Shah model is approximated via an approach originating from the phase field model by Modica and Mortola \cite{MoMo77}.
A simple and widespread level set approach was proposed by Chan and Vese \cite{ChVe01a}.

The goal of this paper is to derive a posteriori error estimates for the binary Mumford-Shah model for fixed constants $c_1$ and $c_2$.
To this end, we proceed as follows:  
We take into account a suitable {\it strictly} convex relaxation of the binary Mumford-Shah functional already studied in \cite{Be10}, which
is related to more general relaxation approaches suggested by Chambolle \cite{Ch05} (see Section \ref{sec:relax}).
For this relaxation, we take into account its predual and set up a corresponding primal-dual algorithm \cite{Ba12, ChPo11,HiKu05} (see Section \ref{sec:primaldual}). 
Then, following Bartels \cite{Ba12}, we use Repin's primal-dual approach \cite{Re00,Re08} to derive functional a posteriori error estimates for the relaxed solution 
based on upper bounds of the duality gap (cf. also the book by Han \cite{Han2005} with respect to mechanical applications) (see Section \ref{sec:APosterioriRelaxed}). 
These estimates can be used together with a suitable cut out argument to derive an a posteriori estimate for the characteristic function $\chi$ minimizing 
the original functional \eqref{eq:OriginalBinaryModel} (see Section \ref{sec:APosterioriChi}).
Moreover, two adaptive finite element discretization schemes and one conventional, non adaptive finite difference scheme are investigated (see Section \ref{sec:FE}).
Finally, we apply the resulting estimate to these schemes incorporating an appropriate post smoothing and present the numerical results (see Sections \ref{sec:algo} and \ref{sec:results}).

\section{A Strictly Convex Relaxation of the Binary Mumford-Shah Model} \label{sec:relax}
Henceforth, we use the notation $\chi_A$ to denote the indicator function of a measurable set $A\subset\Omega$ and define $[u>c]:=\{x\in\Omega:u(x)>c\}$.
Furthermore, we use generic constants $c$ and $C$ throughout this paper.
Rewriting the binary Mumford-Shah functional \eqref{eq:OriginalBinaryModel} as 
\beqn\label{eq:MS}
\EnerMS[\chi,c_1,c_2] =\int_{\Omega}(\ff_1-\ff_2) \chi \differential x +|\D\chi|(\Omega) + \int_{\Omega} \ff_2\differential x\,,
\eeqn
one observes that adding a constant to $\ff_1$ and $\ff_2$ leaves the minimizers $\chi$ unchanged. Thus, we may assume that $\ff_1, \ff_2 \geq c > 0$. 
Now, let us introduce the following relaxed functional 
\beqn
\ERel[\image]=\int_{\Omega}\image^2 \ff_1+(1-\image)^2 \ff_2\differential x+|\D\image|(\Omega)\,,
\label{eq:relaxedMS}
\eeqn
which is supposed to be minimized over all  $\image \in BV(\Omega,\R)$. 
Indeed, $\ERel[\chi] = \EnerMS[\chi,c_1,c_2]$
for characteristic functions $\chi$ and one retrieves the original binary Mumford-Shah model.
Proving existence of minimizers of \eqref{eq:relaxedMS} with the direct method in the calculus of variations is straightforward. 
For details we refer to \cite{AmFuPa00, EvGa92}.
Furthermore, \eqref{eq:relaxedMS} is strictly convex by our above assumptions on $\ff_1$ and $\ff_2$.
Loosely speaking,  a preference for the values $0$ and $1$ for $\image$ is encoded in the quadratic growing data term.
The minimizers of both functionals \eqref{eq:OriginalBinaryModel} and \eqref{eq:relaxedMS} are related in the following sense (cf. \cite{Be10}):
\begin{proposition}[Convex relaxation and thresholding]\label{ConvexRelaxationBinaryMS}
Under the above assumptions, a minimizer $\image\in BV(\Omega)$ of the functional $\ERel$ exists, $\image(x)\in[0,1]$ for a.e. $x\in\Omega$ and
\beq
\chi_{[\image>0.5]}\in\argmin_{\chi\in BV(\Omega,\{0,1\})}\EnerMS[\chi]\,.
\eeq
\end{proposition}
\begin{proof}
Proposition \ref{ConvexRelaxationBinaryMS} is an instance of a more general result, which can be found in \cite{Ch05,ChDa09}.
In fact, let  $\Psi:\Omega\times\R\rightarrow\R$ be measurable, $\Psi(\cdot,t)\in L^1(\Omega)$ for a.e. $t\in\R$,
$\Psi(x,\cdot)\in C^1(\R)$ be strictly convex for a.e. $x\in\Omega$, $\Psi(x,t)\geq c|t|-C$ and
$E_\Psi[\image]=\int_{\Omega}\Psi(x,\image)\differential x+|\D\image|(\Omega)\,$.
Then, there exists a minimizer
$\image\in\argmin_{\tilde\image\in BV(\Omega)}E_\Psi[\tilde\image]$
and for $s\in\R$
\beq
\chi^s:=\chi_{[\image>s]}\in\argmin_{\chi\in BV(\Omega,\{0,1\})}\int_{\Omega}\partial_t\Psi(x,s)\chi\differential x +|\D\chi|(\Omega)\,.
\eeq
With $\Psi(x,t)=t^2\ff_1(x)+(1-t)^2\ff_2(x)$ and $s=\frac12$ this allows to verify the main claim of Proposition~\ref{ConvexRelaxationBinaryMS}. 
Indeed, for $t\in\R$ and $\chi\in BV(\Omega,\{0,1\})$, let $\ERel_t[\chi]:=\int_{\Omega}\partial_t\Psi(x,t)\chi\differential x+|\D\chi|(\Omega)$. 
For our specific choice of $\Psi$, $\ERel_t[\chi]=\int_{\Omega} \left(2t(\ff_1(x)+\ff_2(x)) - 2\ff_2(x)\right)\chi  \differential x + |\D\chi|(\Omega)\,$ implies that
minimizing the functional $\ERel_\frac12$ is equivalent to minimizing the functional $\EnerMS[\cdot,c_1,c_2]$ because
$\ERel_\frac12[\chi] = \EnerMS[\chi,c_1,c_2] -  \int_{\Omega} \ff_2\differential x$.

For the sake of completeness, we give here the proof of the more general statement tailored to our particular instance.
Obviously, $\ERel[\min\{\max\{0,v\},1\}]\leq \ERel[v]$ for any $v\in BV(\Omega)$. 
Hence, the minimizer takes values in the range $[0,1]$.
Due to Fubini's Theorem, we obtain
\[
\int_{\Omega}\Psi(x,\image)\differential x= \int_{\Omega}\Psi(x,0)\differential x +\int_0^1\int_\Omega\partial_t\Psi(x,t)\chi_{[u>t]}\differential x\differential t
\]
for $u\in BV(\Omega,[0,1])$.
For $s\in\R$, let $\chi^s$ denote a minimizer of $\ERel_s$ in the set $BV(\Omega,\{0,1\})$. 
Using the coarea formula (see \cite[Theorem 3.40]{AmFuPa00}) and the minimization property of $\chi^s$ we get
$\ERel[\image] \geq C_\Psi+\int_0^1\EnerS\left[\chi^s\right]\differential s\,,$ 
where $C_\Psi = \int_{\Omega}\Psi(x,0)\differential x$. 

Next, we make use of the following monotonicity result  \cite[Lemma 4]{AlCaCh05}:
Let $h_1,h_2\in L^1(\Omega)$ such that $h_1\leq h_2$ a.e. in $\Omega$ and assume
$\chi_i$ is a minimizer of $\int_{\Omega}h_i\chi\differential x+|\D\chi|(\Omega)$ in  $BV(\Omega,\{0,1\})$
for $i\in\{1,2\}$, then $\chi_1\geq\chi_2$ a.e. in $\Omega$.
Due to the strict positivity of $\ff_1$ and $\ff_2$ we obtain that $\partial_t\Psi(x,s)$ is strictly increasing in $s$ for fixed $x$. Thus, choosing $h_i = \partial_t\Psi(x,s_i)$
we get that $\chi^{s_1} \geq \chi^{s_2}$ a.e. in $\Omega$ for $s_1 \leq s_2$.
This implies that $u^\ast(x) = \sup \{ s \,|\, \chi^s(x) = 1\}$ is well-defined and Lebesgue measurable with 
$\chi^s = \chi_{[u^\ast > s]}$. Furthermore, by the coarea formula we get $\image^\ast \in BV(\Omega)$ and $C_\Psi+\int_0^1 \EnerS[\chi^s]\differential s=\ERel[\image^*]\,$.
Altogether, we obtain the chain of estimates
\beq
\ERel[\image] 
\geq C_\Psi+\int_0^1\EnerS\left[\chi^s\right]\differential s=\ERel[\image^*]\geq\ERel[\image]\,,
\eeq
in which the inequalities are actually equalities.
Thus, for a.e.\ $s\in [0,1]$ the characteristic function $\chi_{[u^\ast>s]}$ is a minimizer of $\EnerS$.
We are left to show that this holds in particular for $s=\tfrac12$.
To this end, we consider a monotonously decreasing sequence $s_n$, which converges to $\tfrac12$.
For any $\tilde\chi\in BV(\Omega,\{0,1\})$ we can infer by the dominated convergence theorem
and the weak-$*$ lower semicontinuity of the total variation (noting that $\chi_{[\image>s_n]}$ converges weak-$*$ in $BV$ to $\chi_{[\image>\frac12]}$) 
\beq
\ERel_{\frac12}[\tilde\chi]=\liminf_{n\rightarrow\infty}\ERel_{s_n}[\tilde\chi]\geq\liminf_{n\rightarrow\infty}\ERel_{s_n}[\chi_{[\image>s_n]}]\geq\EnerS[\chi_{[\image>\frac12]}]\,,
\eeq
from which the assertion follows (indeed the more general result holds by the same argument for all $s\in [0,1)$).
\end{proof}

\section{A Primal Dual Approach for the Relaxed Problem}\label{sec:primaldual}
In this section, we make use of convex analysis to derive a duality formulation for the 
minimization problem of the relaxed functional \eqref{eq:relaxedMS}. Primal and dual formulation will later be used in the a posteriori estimates. The dual of $BV(\Omega)$ is very difficult to characterize and not suitable for computational purposes. Thus, for a generalized ROF model, Hinterm\"uller and Kunisch \cite{HiKu05} proposed to consider the corresponding $BV$ functional as the dual of another functional, which we refer to as the \textit{predual functional}. Bartels \cite{Ba12} made use of this approach in the context of a posteriori estimates for the ROF model. Here, we follow this procedure and investigate the predual of \eqref{eq:relaxedMS}.

Recall that the \textit{Fenchel conjugate} $J^{*}$ of a functional $J: X \to \bar \R$ on a Banach space 
$X$ with $\bar \R = \R \cup \{\infty\}$ is a functional on the dual space $X^\prime$ with values in $\bar \R$, defined as $J^{*}[x'] = \sup_{x\in X}\{\langle x',x\rangle-J[x]\}\,$,
where $\langle\cdot,\cdot\rangle$ denotes the duality pairing.
Furthermore, we denote by $\Lambda^{*}\in\mathcal{L}(Y',X')$ the \textit{adjoint operator} of $\Lambda\in\mathcal{L}(X,Y)$ and by $\partial J$ the \textit{subgradient} of $J$  
(cf.  \cite{EkTe99}).

Now, we investigate an energy functional 
\beqn
\DRel[q] = F[q] + G[\Lambda q]\qquad q\in \Q\,, 
\label{eq:prototypeFunctional}
\eeqn
where $F:\Q\rightarrow\bar\R$ and $G:\V\rightarrow\bar\R$ being proper, convex and lower semicontinuous functionals,  $\V$ and $\Q$ being reflexive Banach spaces
and $\Lambda\in\mathcal{L}(\Q,\V)$.
In our case of the predual of the convex relaxed binary Mumford-Shah model, we have
\beq
 F[q]=\IB[q]=\left\{\begin{array}{c l} 0&\text{if }|q|\leq 1\text{ a.e.}\\
                    +\infty&\text{else}
                   \end{array}
\right.\,,
\quad
G[v]=\int_{\Omega}\frac{\frac{1}{4}v^{2}+ v\ff_2-\ff_1\ff_2}{\ff_1+\ff_2}\differential x\,,
\eeq
with $\Lambda=\div$, $\Q=H_N(\div,\Omega)$ and $\V=L^2(\Omega)$. Recall the definition of the spaces $H(\div,\Omega)=\{q\in L^2(\Omega,\mathbb{R}^n):\ \div q \in L^2(\Omega)\}$,
endowed with the norm $\|q\|_{H(\div,\Omega)}^2=\|q\|_{L^2(\Omega)}^2+\|\div q\|_{L^2(\Omega)}^2$,
and $H_N(\div,\Omega)=H(\div,\Omega)\cap\{q\cdot \nu =0\text{ on }\partial\Omega\}$, where $\nu$ is the outer normal on $\partial \Omega$ and the operator
$\div$ is understood in the weak sense.  Moreover, $\Lambda^{*}=-\nabla$ holds in the sense
\beq
 (\Lambda^{*} v,q)_{L^2(\Omega)}=(v,\div q)_{L^2(\Omega)}\quad \forall v\in\V,\ q\in\Q\,.
\eeq
Based on this duality and for the particular choice of $\DRel$, we easily verify that $(\DRel)^\ast=\ERel$. 
Indeed, from the general theory in \cite[pp. 58 ff.]{EkTe99}, we can deduce $(\DRel)^\ast[v]=F^{*}[-\Lambda^{*}v]+G^{*}[v]$.
As a result of the denseness of $C^1_c(\Omega)$ in $H_N(\div,\Omega)$ with respect to the norm $\|\cdot\|_{H(\div,\Omega)}$, we can infer for any $v\in BV(\Omega)$
\beq
|\D v|(\Omega)=\sup_{q\in\Q,\|q\|_{\infty}\leq 1}\int_{\Omega}v\div q\differential x
=\sup_{q\in\Q}\left(-\int_{\Omega}v\div q\differential x-\IB[q]\right)\,,
\eeq
which leads to
\beq
F^{*}[-\Lambda^{*}v]=\sup_{q\in\Q}\left(-\int_{\Omega}v\div q\differential x-\IB[q]\right)=|\D v|(\Omega)\,.
\eeq
On the other hand, the Fenchel conjugate of $G$ can be computed as follows:
\beq
G^{*}[v]=\sup_{w\in L^2(\Omega)}\left((v,w)_{L^2(\Omega)}-G[w]\right)=\int_\Omega v^2\ff_1+(1-v)^2\ff_2\differential x\,,
\eeq
where the supremum is attained for $w=2v(\ff_1+\ff_2)-2\ff_2$. This verifies the assertion.

Now, the central insight is that 
\begin{equation}\label{eq:duality}
\DRel[p] = -(\DRel)^\ast[u]
\end{equation} 
for a minimizer $p$ of $\DRel$ and a minimizer $u$ of $(\DRel)^\ast$. This can be seen by formally exchanging $\inf$ and $\sup$ as follows
\begin{align*}
 \DRel[p]=&\inf_{q\in H_N(\div,\Omega)}(F[q]+G[\Lambda q])=\inf_{q\in H_N(\div,\Omega)}\sup_{v\in L^2(\Omega)}(F[q]+\langle v,\Lambda q\rangle-G^{*}[v])\\
 =&\!\!\sup_{v\in L^2(\Omega)}\!\left(-\!\!\!\!\!\sup_{q\in H_N(\div,\Omega)}\!\!(\langle -\Lambda^{*}v,q\rangle-F[q])-G^{*}[v]\right) 
\!=\!\!\!\sup_{v\in L^2(\Omega)}\!\!\!\left(-F^{*}[-\Lambda^{*} v]-G^{*}[v]\right)\\
 =&\!\sup_{v\in L^2(\Omega)}(-(\DRel)^\ast[v])=-\inf_{v\in L^2(\Omega)}(\DRel)^\ast[v]=-(\DRel)^\ast[u]\,.
\end{align*}
A rigorous verification can be found in \cite[Chapter III.4]{EkTe99} (see also \cite{Ro97a,Re00,Ba12}). 
 Furthermore,  one obtains that 
$\bar q\in \Q$ and $\bar v\in \V$ are optimal if and only if  $-\Lambda^* \bar v \in\partial F[\bar q]$ and $\bar v\in\partial G[\Lambda\bar q]$,
which can be deduced from the equivalence 
$J[x]+J^{*}[x']=\langle x',x\rangle\Longleftrightarrow x'\in\partial J[x]\,$
(see \cite[Proposition I.5.1]{EkTe99}).

\section{Functional A Posteriori Estimates for the Relaxed Problem}
\label{sec:APosterioriRelaxed}
In what follows, we investigate a posteriori error estimates associated with the energy
$\DRel[q]=F[q]+G[\Lambda q]$ and its dual $\ERel[v]=F^{*}[-\Lambda^{*}v]+G^{*}[v]$. A crucial prerequisite is the uniform convexity of $G$, which is linked to the specific choice of the relaxed Model $\ERel$.

Recall that a functional $J:X\rightarrow\R$ is \textit{uniformly convex}, if
there exists a continuous functional $\Phi_{J}:X\rightarrow[0,\infty)$ such that 
$J\!\left[\tfrac{x_1+x_2}{2}\right]+\Phi_J(x_2-x_1)\leq \tfrac12(J[x_1]+J[x_2])$
for all $x_1,x_2\in X$ and $\Phi_J(x)=0$ if and only if $x=0$. 
Furthermore, we denote by $\Psi_J$ a non-negative functional such that
$\langle x', x_2-x_1\rangle+\Psi_J(x_2-x_1)\leq J[x_2]-J[x_1]$ for all $x'\in\partial J[x_1]\,.$
Hence, $\Psi_J$ allows a quantification of the strict monotonicity of $J$.
If $J\in C^2$ and $\lambda_{min}$ denotes the smallest 
eigenvalue of  $D^2 J[0]$, then $\Phi_J$ and $\Psi_J$ admit the representation
\beq
\Phi_J(x)=\frac{1}{8}\lambda_{min}(D^2 J[0])\|x\|^2\quad\text{and}\quad
\Psi_J(x)=\frac{1}{2}\lambda_{min}(D^2 J[0])\|x\|^2\,,
\eeq
which follows readily via a Taylor expansion.

Now, the a posteriori error estimate is based on the following direct application of a general result by Repin \cite{Re00}:
Let $\image\in\argmin_{\tilde v\in \V}\ERel[\tilde v]$ and $q\in \Q$, $v\in \V^\prime=\V=L^2(\Omega)$. Then,
\beqn
\Phi_{G^*}(u-v)+\Phi_{F^*}(\Lambda^*(u-v))+\Psi_{\ERel}\!\left(\frac{u-v}{2}\right)\leq \frac{1}{2}(\ERel[v]+\DRel[q])\,.
\label{eq:abstractAPosterioriError}
\eeqn
The proof of \eqref{eq:abstractAPosterioriError} relies on the following two estimates:
At first, 
\begin{eqnarray*}
&&\Phi_{G^*}(u-v)+\Phi_{F^*}(\Lambda^*(u-v))\\
&&\leq
\frac{1}{2}\left(F^*[\Lambda^*v]+G^*[v]+F^*[\Lambda^*\image]+G^*[\image]\right)-\left(F^*\!\left[\Lambda^*\frac{\image+v}{2}\right]+G^*\!\left[\frac{\image+v}{2}\right]\right)\,,
\end{eqnarray*}
which is a consequence of the uniform convexity. Secondly, using the monotonicity and that $\image$ is
a minimizer of $E$ and thus $0\in\partial \ERel[\image]$, we get
\beq
\Psi_{\ERel}\!\left(\frac{\image-v}{2}\right)\leq
F^*\!\left[\Lambda^*\frac{\image+v}{2}\right]+G^*\!\left[\frac{\image+v}{2}\right]
-\left(F^*[\Lambda^*\image]+G^*[\image]\right).
\eeq
The claim follows by adding both estimates and using the fundamental relation 
$\ERel[\image]\geq - \DRel[q]$ known as the \textit{weak complementarity principle} (cf. \cite{EkTe99,Re00, Ba12}).
In the case of the binary Mumford-Shah model, we easily compute
\beq
\Phi_{F^{*}}\equiv 0,\ \Phi_{G^{*}}(v)=\frac{1}{4}\int_\Omega v^2(\ff_1+\ff_2)\differential x, \ \Psi_{\ERel}(v)=\int_\Omega v^2(\ff_1+\ff_2)\differential x
\eeq
and the estimate \eqref{eq:abstractAPosterioriError} implies for any $v\in \V$ and $q\in \Q$ 
\beq
\int_{\Omega}(\image-v)^2(\ff_1+\ff_2)\differential x
\leq \ERel[v]+\DRel[q]\,.
\eeq
Finally, $\frac12(a-b)^2\leq a^2+b^2$ with $a=c_1-u_0$ and $b=c_2-u_0$ yields $\frac{1}{2\nu}(c_1-c_2)^2\leq \ff_1+\ff_2$.
Thus, we obtain the following theorem:
\begin{theorem}\label{thm:first}
Let $\image\in\V$ be the minimizer of $\ERel$. Then, for any $v\in \V$ and $q\in \Q$ it holds that
\beqn
\|\image-v\|_{L^2(\Omega)}^2 \leq \err_\image^2[v,q] :=
\frac{2\nu}{(c_1-c_2)^2}\left(\ERel[v]+\DRel[q]\right)\,.
\label{eq:APosterioriRelaxedMS}
\eeqn
\end{theorem}
In the application, one asks for (post processed) discrete primal $v$  and dual solution $q$ which ensure a small right hand side.  
Additionally, the estimator $\err_\image$ is \textit{consistent}, \ie $\err_\image[v,q]\rightarrow 0$ provided $v$ and $q$ converge to the extrema of the corresponding energy functionals \wrt the topology of the associated Banach spaces.

\section{A Posteriori Error Estimates for the Binary Mumford-Shah Model} \label{sec:APosterioriChi}
In the sequel, we expand the a posteriori theory to the binary Mumford-Shah model. The key observation is that for many images approximate solutions $u\in L^2(\Omega)$ of the relaxed model are characterized by steep profiles, where the actual solution of the original binary Mumford-Shah model jumps. 
Thus, we proceed as follows. We define 
\beq
\jumpArea[v,\eta]=\left\|\chi_{[\frac12-\eta\leq v\leq \frac12+\eta]}\right\|_{L^1(\Omega)}
\eeq
for $\eta\in \left(0,\tfrac12\right)\,$,
which measures the area of the preimage of the interval of size $2\eta$ centered at the 
threshold value $s=\tfrac12$ (cf. Section \ref{sec:relax}). Based on the above observation, the
set $\mathcal{S}_\eta=\left[\tfrac12-\eta\leq v\leq \tfrac12+\eta\right]$ can be regarded as the set of non properly identified phase. 
Taking into account this definition, we obtain the following theorem.

\begin{theorem}[A posteriori error estimator for the binary Mumford-Shah model]\label{thm:apost}
For fixed $c_1$ and $c_2$ let $\chi \in BV(\Omega, \{0,1\})$ be a minimizer of the binary Mumford-Shah functional
$\EnerMS[\cdot,c_1,c_2]$ \eqref{eq:OriginalBinaryModel}. Then for 
all $v\in \V=L^2(\Omega)$ and $q\in \Q=H_N(\div,\Omega)$ we have that 
\beqn
\left\|\chi - \chi_{[v>\frac12]}\right\|_{L^1(\Omega)}\leq \err_\chi[v,q] 
:= \inf_{\eta \in (0,\frac12)}\left(\jumpArea[v,\eta]+\frac{1}{\eta^2}\err_u^2[v,q]\right)\,.
\label{eq:errorEstimatorBinary}
\eeqn
\end{theorem}
Let us remark that $\chi_{[v>\frac12]}$ is the result of the same thresholding, which relates $\chi$ to the solution $u$ of the relaxed problem \eqref{eq:relaxedMS}, \ie $\chi=\chi_{[u>\frac12]}$, this time applied to $v$. 
\begin{proof}
Recall that any minimizer $\image$ of the $\ERel$ fulfills $0\leq\image\leq 1$. For all $\eta \in (0,\tfrac12)$ we obtain the following set relation for the symmetric difference of the sets 
$[\image>\frac12]$ and $[v>\frac12]$ ($\Delta$ denoting the symmetric difference of two sets):
\beq
\left[\image>\tfrac12\right]\Delta\left[v>\tfrac12\right] \subseteq\left\{x \in \Omega\,\big|\, \tfrac12-\eta\leq v(x)\leq\tfrac12+\eta\right\}\cup\left\{x\in \Omega \,\big|\, |\image-v(x)|>\eta\right\}\,.
\eeq
Now, using Theorem \ref{thm:first} the Lebesgue measure of the rightmost set can be estimated as follows
\beq
\Lebesgue^n(|\image-v|>\eta)\leq \int_{\{|\image-v|>\eta\}}\frac{|\image-v|^2}{\eta^2}\differential x\leq \frac{1}{\eta^2}\err_\image^2[v,q]\,,
\eeq
where $\eta\in(0,\frac12)$. Finally, taking the infimum for all  $\eta \in (0,\tfrac12)$ concludes the proof.
\end{proof}

In the application, the computational cost to find the optimal $\eta$ is of the order of the degrees of freedom for the discrete solution and thus affordable. 
Let us emphasize that the error estimator $\err_\chi$ is not tailored to a specific finite element approach. Indeed, we can project any primal and dual solution 
onto the spaces $\V = L^2(\Omega)$ and $\Q=H_N(\div,\Omega)$, respectively. 
We will exploit this in the next sections.

\section{Finite Element and Finite Difference Discretization}\label{sec:FE}
In this section, we investigate different numerical approximation schemes for the primal and the dual solution of the relaxed problem \eqref{eq:relaxedMS}
on adaptive meshes and the refinement of the meshes based on the a posteriori error estimate in Theorem \ref{thm:apost}. 
In the context of image processing applications with input images usually given on a regular rectangular mesh, 
an adaptive quadtree for $n=2$ (or octree for $n=3$) turned out to be an effective choice for an adaptive mesh data structure. 
In what follows, we pick up the finite element approach for a variational problem on $BV$ proposed by Bartels \cite{Ba13} and a simplified version of the latter. 
Furthermore, we consider the widespread finite different scheme proposed by Chambolle \cite{Ch04}.

\paragraph{(FE) Finite element scheme on an induced adaptive triangular grid.}
We consider $\Omega=[0,1]^2$ in all numerical experiments in this paper. 
On this domain, we consider 
an adaptive mesh $\quadMesh$ described by a quadtree with cells $\mathscr{C}\in\quadMesh$ being squares, which are recursively refined into four squares via an edge bisection.
We suppose that the level of refinement between cells at edges differs at most by one. 
Thus, on a single edge at most one hanging node appears. Let $h$ indicate the spatially varying mesh size function on $\Omega$, 
\ie in the range of an initial mesh size $2^{-L_{\mbox{\tiny init}}}$ and a finest mesh size $2^{-L_0}$ (usually determined by the image resolution). 
For all discretization approaches investigated here, the degrees of freedom are associated with the non hanging nodes.
Let us denote by $N_v$ the number of these nodes, which will coincide with the number of degrees of freedom of discrete primal functions.
The finite element discretization is based on a triangular mesh $\simplexMesh$ spread over the adaptive quadtree mesh via a
splitting of each quadratic leaf cell  into simplices $\mathscr{T}$ ( ``cross subdivision''). 
We ask for discrete primal functions $u_h$ in the space of piecewise affine and globally continuous functions on $\simplexMesh$ 
denoted by $\mathcal{V}_h$.  Thus, for functions $v_h \in \mathcal{V}_h$ the values at hanging nodes are interpolated based on the values at adjacent non hanging nodes, 
which are associated with the actual degrees of freedom. By $\Q_h = \left\{q_h\in \mathcal{V}^N_h :\ q_h\cdot \nu=0\text{ on }\partial\Omega\right\}$ we denote the discrete counterpart of $\Q$. To accommodate this boundary condition, the boundary nodes are modified after each update of the dual solution
in a post-processing step. 
On $\mathcal{V}_h$, we define discrete counterparts of the continuous functionals $F$ and $G$ as follows:
\begin{eqnarray*}
G_h[v_h] := \int_\Omega  \frac{\frac14 v_h^2 + v_h\theta_{2,h} - \theta_{1,h} \theta_{2,h}}{\theta_{1,h} + \theta_{2,h}} \differential x\,, \quad
F_h[q_h] :=  I_{\bar B_1}[q_h]\,,
\end{eqnarray*}
where $\theta_{i,h}= \mathcal{I}_h(\theta_{i})=\mathcal{I}_h(\tfrac{1}{\nu}(c_i-u_0 )^2)$ for $i=1,2$
with $\mathcal{I}_h$ denoting the Lagrange interpolation. 
In the application on images, we suppose that $u_0 \in \mathcal{V}_0$, where $\mathcal{V}_0$ is the simplicial finite element space corresponding to the full resolution image on the finest grid level 
$L_0$ representing the full image resolution.
Furthermore, we consider two different scalar products. On $\V_h$, we take into account the $L^2$-product and on $\Q_h$ the lumped mass product 
$(q_h,p_h) \mapsto \int_\Omega \mathcal{I}_h(q_h p_h) \differential x$ and identify $\V_h$ and $\Q_h$ with their dual spaces with respect to the $L^2$- and the lumped mass product, respectively.
Then, the associated dual operators are 
\begin{eqnarray*}
G_h^\ast[v_h] = \int_\Omega v_h^2 \theta_{1,h} + (1-v_h)^2 \theta_{2,h} \differential x \,, \quad 
F_h^\ast[q_h] =  \int_\Omega \mathcal{I}_h(|q_h|) \differential x\,.
\end{eqnarray*}
Finally, we define the discrete divergence $\Lambda_h: \Q_h \to \V_h$, $q_h \mapsto \proj\div q_h$, where $\proj$ denotes the $L^2$-projection $\proj:L^2(\Omega) \rightarrow \mathcal{V}_h$.
Following Bartels \cite{Ba12} and taking into account the above scalar products on $\V_h$ and on $\Q_h$,
we obtain for the discrete gradient $-\Lambda_h^\ast:\V_h \to \Q_h$, $v_h \mapsto -\Lambda_h^\ast v_h$, the defining duality
\begin{equation}
\int_\Omega \mathcal{I}_h(-\Lambda_h^\ast v_h \cdot q_h) \differential x = \int_\Omega v_h \proj \div q_h \differential x
\label{eq:definitionLumpedGradient}
\end{equation}
for all $q_h\in \Q_h$ and $v_h \in \V_h$.

\paragraph{(FE') Finite element scheme based on a simple gradient operator.}
Instead of the above defined discrete gradient operator $-\Lambda_h^\ast$, we alternatively consider the 
piecewise constant gradient $\nabla v_h$  on the simplices $\mathscr{T}$ of the simplicial mesh for functions $v_h \in \V_h$. To this end, we 
choose $\Q_h$ as the space of piecewise constant functions on the simplicial mesh, and take into account the standard $L^2$-product on both spaces.
The above definitions of the functionals $G_h$ and $F_h$ are still valid. Moreover, $G^\ast_h$ remains the same, only $F^\ast_h$  changes to 
$F_h^\ast[q_h] =  \int_\Omega |q_h| \differential x$.  The discrete divergence $\Lambda_h: \Q_h \to \V_h$ is defined via duality starting from the preset discrete gradient as  
\[
\int_\Omega \mathcal{I}_h(\Lambda_h q_h v_h) \differential x = - \int_\Omega q_h  \cdot \nabla v_h \differential x\,,
\]
which indeed ensures that $-\Lambda_h^\ast v_h = \nabla v_h$. This simplified ansatz leads to a non-conforming iterative solution scheme (see Section \ref{sec:algo}),
since the space of piecewise constant finite elements is not contained in $H_N(\div,\Omega)$ (cf. \cite{Ba13}). 
After each modification of the (piecewise constant) dual solution the values on the corresponding boundary cells are set to $0$ to satisfy the boundary condition.
To apply the derived a posteriori error estimates a projection onto the space $H_N(\div,\Omega)$ is required.
To this end, we replace the solution $p_h\in \Q_h$  by its $L^2$-projection onto the space $\V_h^n$.

\paragraph{(FD) Finite difference scheme on a regular mesh.} 
The finite difference scheme for the numerical solution of functionals on $BV$ proposed by Chambolle \cite{Ch04}
is extensively used in many computer vision applications and applies to image data defined on a structured non adaptive mesh. 
We compare the a posteriori error estimator for this scheme on non adaptive meshes with the above finite element schemes on adaptive meshes. 
To this end, we denote by $\Vh \in \R^{N_v}$ and $\Qh \in \R^{2N_v}$ nodal vectors on the regular lattice for primal and dual solutions, respectively.
Here, $N_v = (h^{-1}+1)^2$, where  $h$ denotes  the fixed grid size of the finite difference lattice.
Integration is replaced by summation and we obtain the following discrete analogues $\Gh$ and $\Fh$ of the continuous functionals $F$ and $G$ as functions on 
$\R^{N_v}$ and $\R^{2 N_v}$, respectively:
\begin{eqnarray*}
\Gh[\Vh] := \sum_{i=1}^{N_v}  \left( \frac{\frac14 (\Vh^i)^2 + \Vh^i\ffhtwoi - \ffhonei \ffhtwoi}{\ffhonei + \ffhtwoi} \right)\,, \quad
\Fh[\Qh] :=  \max_{i=1,\ldots, N_v} \mathbf{I}_{\bar B_1}[\Qh^i]
\end{eqnarray*}
with $\ffhonei$, $\ffhtwoi$ denoting the pointwise evaluation of $\theta_1$ and $\theta_2$, respectively, and $\mathbf{I}_{\bar B_1}[\Qh^i]=0$ for $|\Qh^i|\leq 1$ and $+\infty$ otherwise.
The associated dual operators for the standard Euclidean product as the duality pairing are 
\begin{eqnarray*}
\Ghstar[\Vh] = \sum_{i=1}^{N_v}   (\Vh^i)^2 \ffhonei + (1-\Vh^i)^2 \ffhtwoi \,, \quad
\Fhstar[\Qh] = \sum_{i=1}^{N_v} \left|\Qh^i\right|\,.
\end{eqnarray*}
Finally, we take into account periodic boundary conditions (by identifying degrees of freedom on opposite boundary segments) and 
use forward difference quotients to define the discrete gradient operator $-\mathbf{\Lambda}_h^\ast : \R^{N_v} \to \R^{2 N_v}$, \ie
\beq
((-\mathbf{\Lambda}_h^\ast) \Vh)^i = \left(\frac{\Vh^{\neigh(i,j)}-\Vh^i}{h} \right)_{j=1,2}\,,
\eeq
where $\neigh(i,j)$ is the index of the neighboring node in direction of the $j$th coordinate vector. 
As a consequence, the matrix representing the discrete divergence operator $\mathbf{\Lambda}_h : \R^{2 N_v} \to \R^{N_v}$ is just the  negative
transpose of the matrix representing the discrete gradient and thus corresponds to a discrete divergence based on backward difference quotients.

To use the a posteriori error estimate in the finite difference context, we consider as a simplest choice the piecewise bilinear functions $u_h$ and $p_h$ 
uniquely defined by the solution vectors $\Uh$ and $\Ph$, respectively.
The boundary condition is taken care of in exactly the same way as in the case (FE).

\section{Implementation based on a Primal-Dual Algorithm}\label{sec:algo}
For the numerical solution of the different discrete variational problems, we use the primal-dual 
algorithm proposed by Chambolle and Pock \cite[Algorithm 1]{ChPo11}, which computes
both a discrete primal and a discrete dual solution to be used in the a posteriori error estimates.
Note that we use \cite[Algorithm 1]{ChPo11} instead of \cite[Algorithm 2]{ChPo11} even though $\Ghstar$ is uniformly convex. 
As we will see later, evaluating $(\Id+\tau\partial\Ghstar)^{-1}$ requires the inversion of a matrix depending on $\tau$. 
In Algorithm~1, $\tau$ is fixed and the inverse can be computed once using a Cholesky decomposition for the 
sparse, symmetric and positive-definite matrix (for details see \cite{ChDaHa09}), 
while in Algorithm~2 the decomposition of the linear system has to be performed in each iteration.
Before we discuss this algorithm in the more conventional matrix-vector notation, let us rewrite the finite element approaches correspondingly.
Let $N_v= \mathrm{dim}\, \V_h$ (the number of non hanging nodes) and $N_q = \mathrm{dim}\, \Q_h$ (for (FE) $N_q = 2 N_v$ and for (FE') $N_q$ is $2$ times the number of simplices).
In what follows, we will use uppercase letters to denote a vector of nodal values, \eg $\Vh^i = v_h(X^i)$ if $X^i$ is the $i$th non hanging node.
The two scalar products are encoded via mass matrices. 
Here, $\Mass \in \R^{N_v,N_v}$ represents the standard $L^2$-product on $\V_h$ and is given by 
 $\Mass\Vh\cdot\Uh = \int_\Omega v_h u_h \differential x$  for all  $v_h,\, u_h \in \mathcal{V}_h$. 
Furthermore, $\tildeMass\in \R^{N_q,N_q}$ is the mass matrix associated with the space $\Q_h$. For the approach (FE)  it is given as the lumped mass matrix with 
$\tildeMass\Ph\cdot\Qh = \int_\Omega \mathcal{I}_h (p_h \cdot q_h) \differential x$ for all $p_h,\, q_h \in \Q_h$, 
whereas for the discretization (FE')  $\tildeMass\Ph\cdot\Qh = \int_\Omega p_h \cdot q_h \differential x$ 
for all $p_h,\, q_h \in \Q_h$ defines a classical (diagonal) mass matrix. 
For the matrix representations $\mathbf{\Lambda}_h$ and $-\mathbf{\Lambda}_h^\ast$ of the discrete divergence and the discrete gradient, respectively, 
we obtain the relation (cf. \cite{Ba12})
\beqn
\mathbf{\Lambda}_h^\ast = \tildeMass^{-1} \mathbf{\Lambda}_h^T \Mass\,.
\label{eq:adjointRelation}
\eeqn
For the discretization (FD) we have $\mathbf{\Lambda}_h^\ast =\mathbf{\Lambda}_h^T$.
Altogether, the discrete predual energy $\DRelh: \R^{N_q} \to \R$ and the discrete energy $\ERelh:\R^{N_v}\to\R$ are defined as follows:
\beq
\DRelh[\Qh] = \Fh[\Qh]+ \Gh[\mathbf{\Lambda}_h \Qh] \,,
\quad 
\ERelh[\Vh] = \Fhstar[-\mathbf{\Lambda}_h^\ast\Vh]+ \Ghstar[\Vh] \,.
\eeq
In the case of  both finite element  schemes $\Gh, \Fh, \Ghstar$ and $\Fhstar$ are defined using the corresponding functions on the finite element spaces, e.g. $\Ghstar[\Vh]:=G_h^\ast[v_h]$.
Now, we are in the position to formulate the primal-dual algorithm.
For a fixed mesh and initial data $(\Uh^0,\Ph^0)\in \R^{N_v} \times \R^{N_q}$
the Algorithm \ref{algo:ChambollePock} proposed by Chambolle and Pock \cite[Algorithm 1]{ChPo11} computes a sequence $(\Uh^k,\Ph^k)$,
which converges to the tuple $(\Uh,\Ph)$ of the discrete primal and dual solution provided $\tau\sigma\|\mathbf{\Lambda}_h\|^2<1$.
\begin{algorithm}[htb]
$k=0$\;
\While{$\|\Uh^{k+1}-\Uh^{k}\|_{\infty}>$THRESHOLD}{
$\Ph^{k+1}=(\Id+\sigma\partial \Fh)^{-1}(\Ph^k-\sigma\Lhstar\barUh^k)$\;
$\Uh^{k+1}=(\Id+\tau\partial \Ghstar)^{-1}(\Uh^k+\tau \Lh\Ph^{k+1})$\; 
$\barUh^{k+1}=2\Uh^{k+1}-\Uh^k$\;
$k=k+1$\;
}
\caption{The primal-dual algorithm used to minimize $\ERelh$.}
\label{algo:ChambollePock}
\end{algorithm}
Indeed, using inverse estimates for finite elements (see \cite{RiWa10} for a computation of the constants) the operator norm
can be bounded in the case (FE') as follows: $\|\mathbf{\Lambda}_h\|^2\leq 48(3+2 \sqrt{2})h_{min}^{-2}\approx 279.8 \,h_{min}^{-2}$,
where $h_{min}$ denotes the minimal mesh size occurring in $\quadMesh$.
Moreover, to estimate the operator norm for the case (FE) we use \eqref{eq:definitionLumpedGradient}
and obtain 
\begin{align*}
\|\mathbf{\Lambda}_h\|^2=& 
\left(\max_{v_h\in\V_h,\|v_h\|_{L^2}=1}\ \max_{q_h\in\Q_h,\|q_h\|_{L^2}=1}
\int_\Omega \mathcal{I}_h(-\Lambda_h^\ast v_h \cdot q_h) \differential x\right)^2\\
=&\left(\max_{v_h\in\V_h,\|v_h\|_{L^2}=1}\ \max_{q_h\in\Q_h,\|q_h\|_{L^2}=1}\int_\Omega v_h \proj \div q_h \differential x\right)^2\\
\leq& \max_{\substack{q_h\in\Q_h,\\\|q_h\|_{L^2}=1}}\|\proj \div q_h\|_{L^2(\Omega)}^2\leq \max_{\substack{q_h\in\Q_h,\\\|q_h\|_{L^2}=1}}\|\div q_h\|_{L^2(\Omega)}^2
\leq 96(3+2 \sqrt{2})h_{min}^{-2}\,.
\end{align*}
Finally, following \cite{Ch04} we can estimate $\|\mathbf{\Lambda}_h\|^2\leq 8h_{min}^{-2}$ for the discretization (FD).

Suitable stopping criteria are a threshold on the primal-dual gap $\ERelh[\Uh^k]+\DRelh[\Ph^k]$ or on the 
maximum norm of the difference of successive solutions $\Uh^{k+1}-\Uh^{k}$ (which we apply here).
To compute the resolvents $(\Id+\partial \Fh)^{-1}[\Qh]$ and $(\Id+\partial \Ghstar)^{-1}[\Vh]$ 
we use a variational ansatz (for details see \cite{Ro97a}), 
\ie for the resolvent of a subdifferentiable functional $J$ with an underlying scalar product $(\cdot, \cdot)$ we have that
\[(\Id+\tau \partial J)^{-1}[x]= \argmin_{y} (x-y,x-y) + 2 \tau J(y)\,.\]
The resolvent of $\Fh$ for the approaches (FE) and (FD) is given by
\beq
(\Id+\sigma \partial\Fh)^{-1}[\Qh]=\left(\tfrac{\Qh^i}{\max\left\{\left|\Qh^i\right|,1\right\}}\right)_{i=1,\ldots N_v}
\eeq 
with $\Qh^i = q_h(X^i)$. In the case (FE'), the above evaluation is performed on each cell.
For the discretizations (FE) and (FE'),  we denote by $\Mass[\Wh]\Uh\cdot\Vh=\int_{\Omega}w_h\,u_h v_h \differential x$ 
the weighted mass matrix for functions $u_h,v_h\in\V_h$ and weight $w_h\in\V_h$.
Then, the resolvent of $\Gh$ reads as
\beq
(\Id+\tau\partial\Ghstar)^{-1}[\Vh]=\left(\Mass\!\left[1+2\tau(\ffhone+\ffhtwo)\right]\right)^{-1}\Mass\left(\Vh+2\tau\ffhtwo\right)\,.
\eeq
In the case (FD), the resolvent is given by
\beq
(\Id+\tau\partial\Ghstar)^{-1}[\Vh^i]=\frac{\Vh^i+2\tau\ffhtwoi}{1+2\tau\left(\ffhonei+\ffhtwoi\right)}\qquad\text{for }1\leq i \leq N_v\,.
\eeq

\noindent  In our numerical experiments, we have chosen  
THRESHOLD $ = 10^{-7}$. For the application considered here, this algorithm turned out to be about $20 \% - 30 \%$ faster than the 
alternating descent method for the Lagrangian for instance used by Bartels \cite[Algorithm A']{Ba12}.

The adaptive mesh refinement is implemented as follows. Given a mesh and initial data for the primal and dual solution, we run the above algorithm and compute the relaxed discrete primal-dual solution pair
$(u_h, p_h)$. In case of the finite difference approach (FD), we define them as the multilinear interpolation on the cells $\mathscr{C}$ 
of the regular mesh.
The corresponding discrete solution of the original problem \eqref{eq:MS} is then given as 
$\chi_h = \chi_{[u_h > \frac12]}\,.$
Based on $u_h$ and $p_h$, we evaluate the local error estimator for every cell $\mathscr{C}_0$ of the full resolution image grid as follows:
\begin{align*}
\err_{\image,\mathscr{C}_0}^2[u_h,p_h] := \frac{2\nu}{(c_1-c_2)^2}\left(\vphantom{\frac{\frac{1}{4}(\div p_h)^{2}+ \div p_h \ff_2-\ff_1\ff_2}{\ff_1+\ff_2}}\right.
\int_{\mathscr{C}_0} &u_h^2 \ff_1+(1-u_h)^2 \ff_2 + |\nabla u_h|   \\
&\left.+ \frac{\frac{1}{4}(\div p_h)^{2}+ \div p_h \ff_2-\ff_1\ff_2}{\ff_1+\ff_2} \differential x\right) \,.
\end{align*}
To this end, a higher order Gaussian quadrature is used. In fact, for (FE) and (FE') we use a Gaussian quadrature of order $4$  
on the simplices $\mathscr{T}_0$ composing the cell $\mathscr{C}_0$ on the finest mesh with full image resolution, where the $\theta_i$ ($i=1,2$) are originally defined,
and for (FD) a Gaussian quadrature of order  $5$ directly on the cells $\mathscr{C}_0$.
The resulting local error estimator for a cell $\mathscr{C}\in\quadMesh$ and the global estimator are given as
\[
\err_{\image,\mathscr{C}}^2[u_h,p_h]  = \sum_{\mathscr{C}_0 \subset \mathscr{C}} \err_{\image,\mathscr{C}_0}^2[u_h,p_h]
\quad\text{and}\quad
\err_{\image}^2[u_h,p_h]  = \sum_{\mathscr{C} \in \quadMesh} \err_{\image,\mathscr{C}}^2[u_h,p_h]\,,
\]
respectively.
We mark those cells $\mathscr{C}$ for refinement for which 
\[
\err_{\image,\mathscr{C}}^2[u_h,p_h] \geq \alpha \max_{\mathscr{C}'\in\quadMesh}\err_{\image,\mathscr{C}'}^2[u_h,p_h]\,,
\]
where $\alpha$ is a fixed threshold in $(0,1)$.  Since this method is prone to outliers,
we additionally sort all local estimators $\err_{\image,\mathscr{C}}^2$ according to their size (starting with the smallest) and mark the cells 
in the upper decile for refinement as well. For the input data from Figure \ref{fig:inputImage} we refine up to the resolution of the initial image.

\section{Numerical Results}\label{sec:results}
In what follows, we show numerical results for four different input images shown in Figure~\ref{fig:inputImage}.
Prior to executing Algorithm \ref{algo:ChambollePock}, we choose suitable values for $c_1$ and $c_2$ by applying 
Lloyd's Algorithm (see \cite{Ll82})
for the  computation of a $2$-means clustering (with initial values $0$ and $1$).  
The resulting values are given in Figure~\ref{fig:inputImage} together with the 
parameters values for $\nu$.
\begin{figure}[htb]
\centerline{
\footnotesize 
\begin{tabular}{lcccc}
&\includegraphics[width=0.195\linewidth]{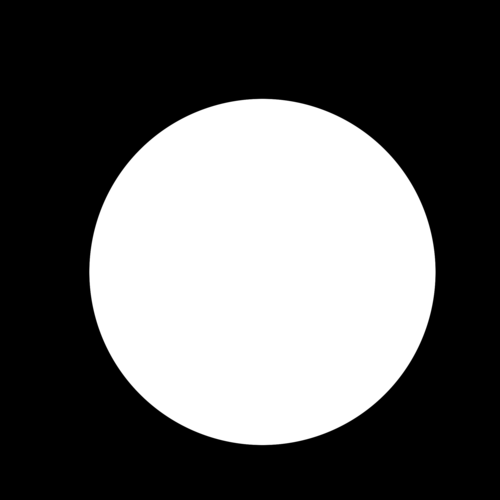} &
\includegraphics[width=0.195\linewidth]{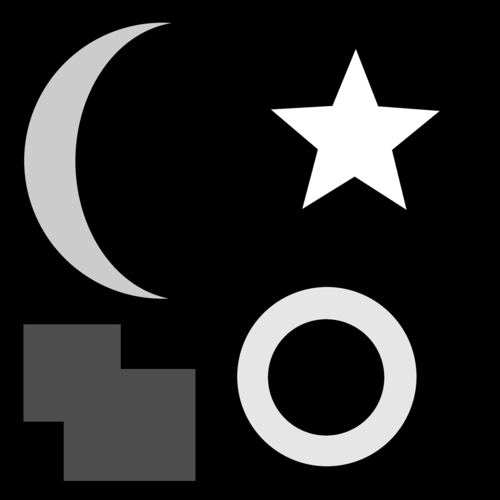} &
\includegraphics[width=0.195\linewidth]{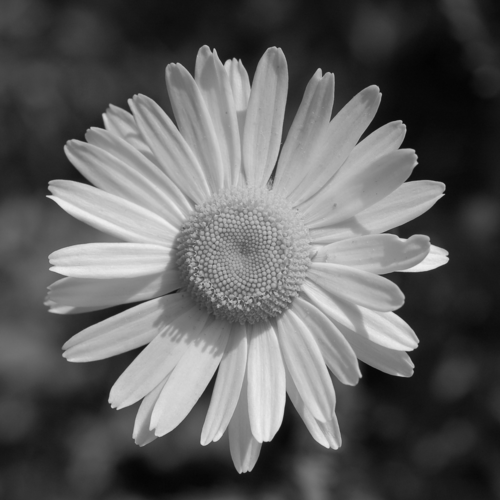} &
\includegraphics[width=0.195\linewidth]{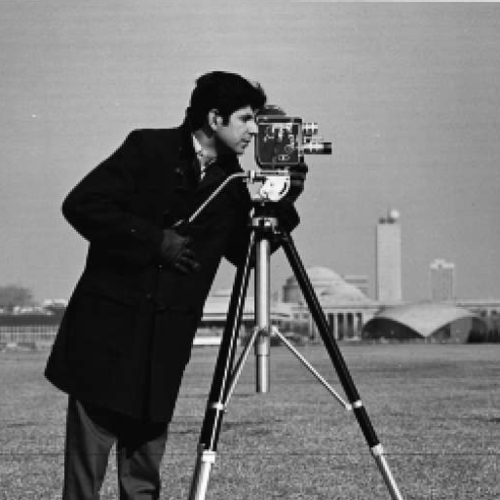} \\
image & (a) & (b) & (c) & (d) \\
resolution & $2049\times2049$ & $2049\times2049$ & $2049\times2049$ & $513\times513$\\
$c_1$ & $0.999772$ & $0.893734$ & $0.664404$ & $0.602566$\\
$c_2$ & $1.99\cdot 10^{-4}$ & $0.030416$ & $0.167763$ & $0.092273$\\
$\nu$ & $5\cdot10^{-3}$ & $5\cdot10^{-3}$ & $10^{-3}$ & $5\cdot10^{-3}$
 \end{tabular}
 }
 \caption{Input images together with the corresponding image resolution and the model parameters $c_1$, $c_2$, and $\nu$
 (flower image: photo by Derek Ramsey, Chanticleer Garden, cameraman image: copyright by Massachusetts Institute of Technology).}
 \label{fig:inputImage}
\end{figure}
The pixels of the input images are interpreted as nodal values of the function $u_0$ on a uniform mesh with 
mesh size $h=2^{-L_0}$ ($L_0 = 11\; ((a),(b),(c)),\; 9 \;(d)$). The algorithm is then started on a uniform mesh of mesh size $h=2^{-L_{\mbox{\tiny init}}}$
($L_{\mbox{\tiny init}} = 5\; ((a),(b),(c)),\; 3 \;(d)$). In all computations we use $\alpha=0.2$, $\tau=10^{-5}$ and $\sigma=5\cdot 10^{-5}$.
We perform $10$ cycles of the adaptive algorithm and refine cells until the depth $L_0$ of the input image is reached.

Applying the primal-dual algorithm, we observe local oscillations for both finite element approaches (FE) and (FE'),
which deteriorate the result of the a posteriori estimator (cf. the numerical results in \cite{Ba12}).
Thus, in a post-processing step, we compensate these oscillations prior to the evaluation of the estimator by an 
application of a smoothing filter.
The filter is defined
via an implicit time step of the discrete heat equation using affine finite elements on the underlying adaptive mesh,
\ie we apply the operator $(\Mass+\iota\Stiff)^{-1}\Mass$ to the solutions, where $\Stiff$ denotes the stiffness matrix.
For the discretization (FE), we choose $\iota=c\cdot h_{min}^2$, where
$h_{min}$ denotes the minimal mesh size of the current adaptive grid, with
$c=3$ and $c=6$ for the primal and the dual solution, respectively.
Moreover, in the case (FE') the smoothing is only applied to the dual solutions
with parameter $\iota=0.75~\cdot~h_{a}^{0.9}$, where $h_{a}$ denotes the average cell size on the adaptive mesh.
In our experiments we observed that these smoothing methods and parameters outperformed other tested choices for the corresponding discretizations.
We call the resulting postprocessed functions $\bar u_h$ and $\bar p_h$, respectively, and replace the local error estimator by
$\err_{\image,\mathscr{C}}^2[\bar u_h,\bar p_h]$.

\begin{table}[htb]
\begin{center}
\begin{tabular}{r l|c | c | c | c }
& & (a) & (b) & (c)& (d) \\ \hline 
\multirow{3}{*}{$\frac{2\nu}{(c_1-c_2)^2}E[u_h]$} &
(FE) & 0.022429 & 0.078497 & 0.124740 & 0.204202 \\ 
&(FE')& 0.022256 & 0.078012 & 0.122620 & 0.202534 \\ 
&(FD) & 0.022493 & 0.078814 & 0.122777 & 0.206166\\
\hline
\multirow{3}{*}{$\frac{2\nu}{(c_1-c_2)^2}D[p_h]$} &
(FE) & -0.021735 & -0.075977 & -0.117259 & -0.182598 \\ 
&(FE')& -0.020950 & -0.070986 & -0.110463 & -0.165980 \\ 
&(FD) & -0.021520 & -0.075455 & -0.119865 & -0.182300\\ \hline
\multirow{3}{*}{$\err_\image^2$} &
(FE) & 0.000694 & 0.002521 & 0.007481 & 0.021603 \\ 
&(FE') & 0.001306 & 0.007025 & 0.012156 & 0.036554 \\ 
&(FD) & 0.000973 & 0.003359 & 0.002912 & 0.023866 \\
\hline
\multirow{3}{*}{$\eta_{optimal}$} &
(FE) & 0.39 & 0.3275 & 0.2825 & 0.305 \\ 
&(FE') & 0.45 & 0.3675 & 0.29 & 0.3925 \\ 
&(FD) & 0.4375 & 0.345 & 0.24 & 0.305\\
\hline
\multirow{3}{*}{$\err_\chi$} &
(FE) & 0.008904 & 0.038779 & 0.175125 & 0.393884 \\ 
&(FE') & 0.009952 & 0.071348 & 0.236609 & 0.65235\\ 
&(FD) & 0.008223 & 0.0425225 & 0.109039 & 0.413503\end{tabular}
\caption{Rescaled dual and primal energy evaluated on the discrete solution $(u_h,p_h)$, error estimator for the relaxed solution, optimal threshold $\eta_{optimal}$  
computed for $u_h$ and resulting a posteriori estimator $\err_\chi$ for the $L^1$-error of the characteristic function $\chi$ (after 10 cycles of the adaptive algorithm).}
\label{table:finalResults}
\end{center}
\end{table}
\begin{figure}[htb]
 \includegraphics[width=0.49\linewidth]{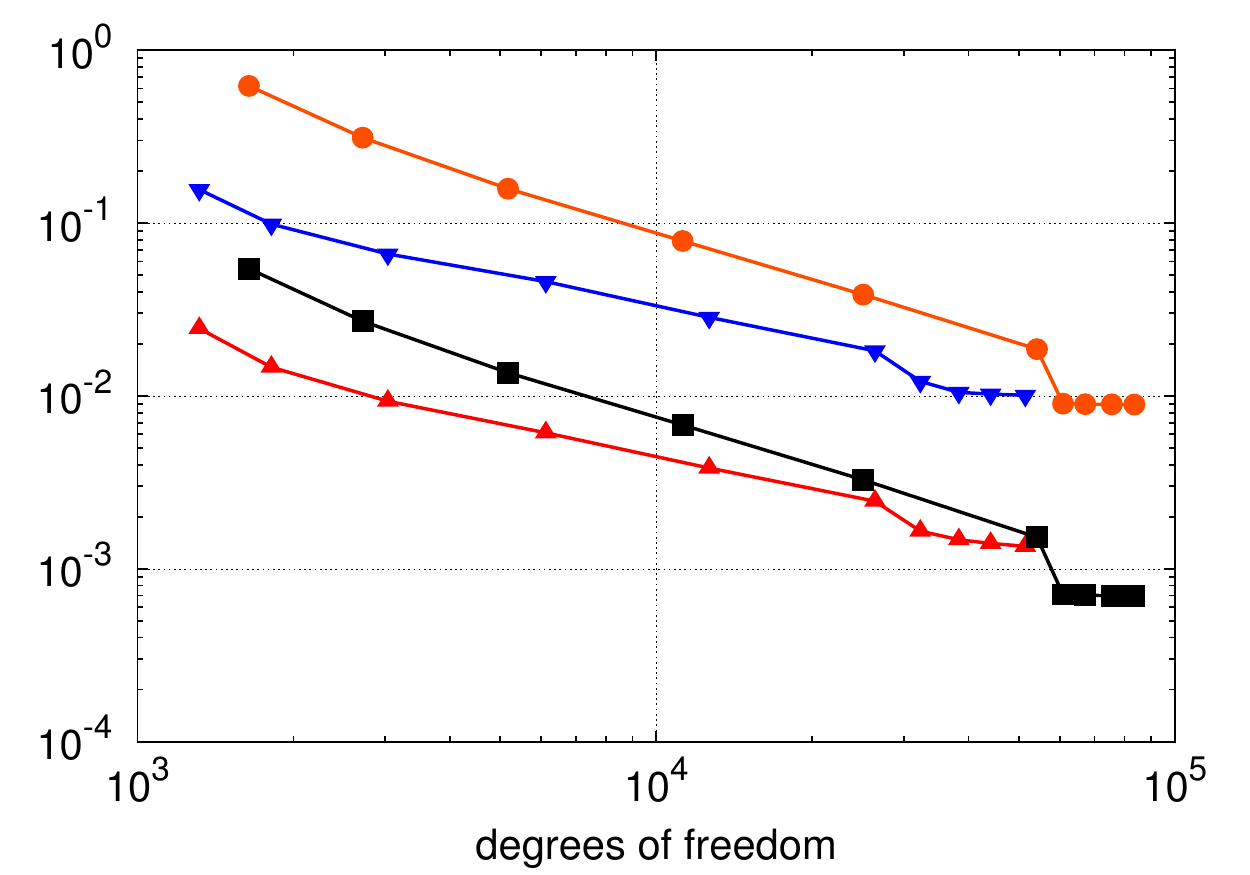}
 \hfill
 \includegraphics[width=0.49\linewidth]{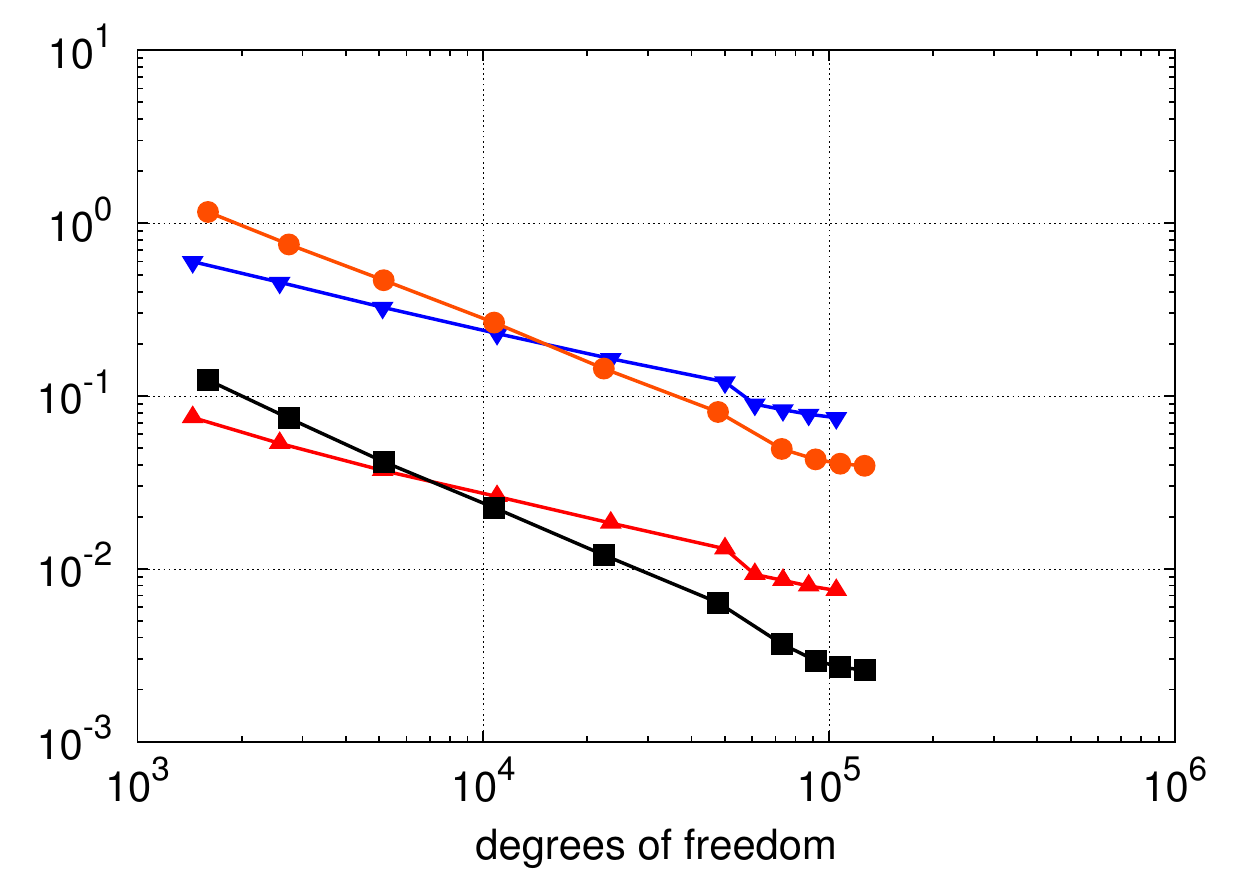}
 \hfill
 \includegraphics[width=0.49\linewidth]{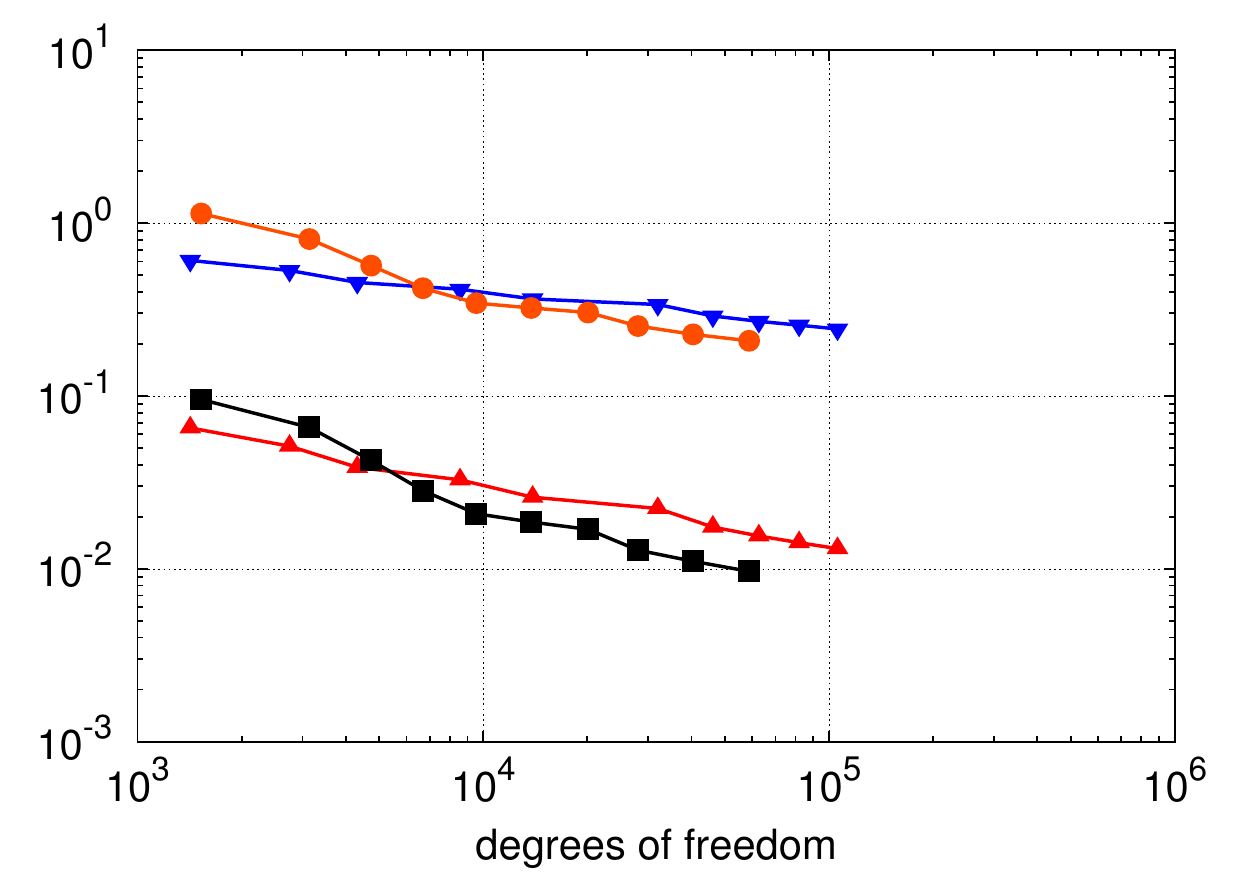}
  \hfill
 \includegraphics[width=0.49\linewidth]{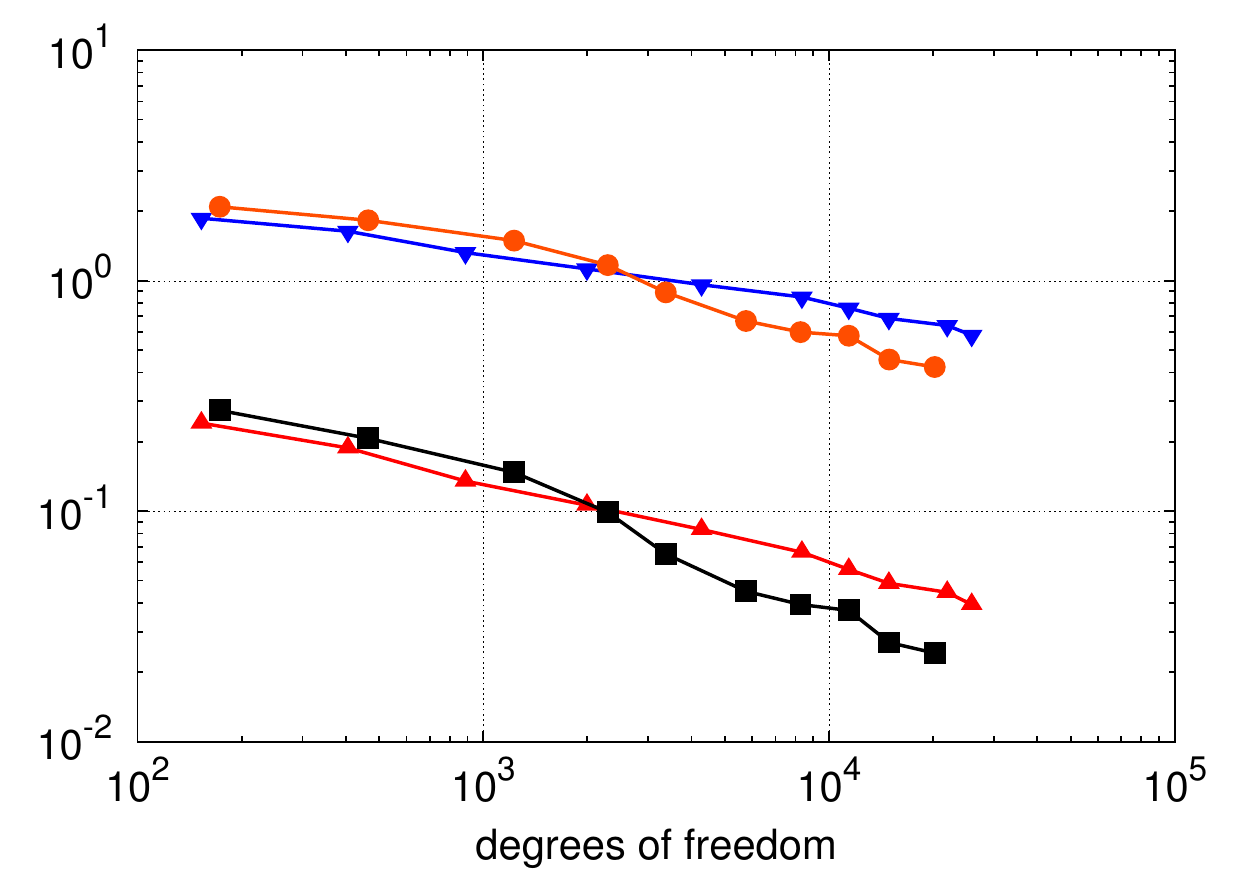}
 \caption{The values of $\err_\image^2$ and $\err_\chi$ are displayed in relation to the number of degrees of freedom  in a log-log plot for the applications (a) (upper left), (b) (upper right), (c) (lower left) and (d) (lower right). The plotted error estimator values correspond to the discretizations (FE) ($\err_\image^2$ (black line, {\color{black} $\blacksquare$}) and $\err_\chi$ (brown line, {\color{gnuplotbrown} $ \bullet$})) and (FE') ($\err_\image^2$ (red line, {\color{gnuplotred} $\filledmedtriangleup$}) and $\err_\chi$ (blue line, {\color{gnuplotblue} $ \filledmedtriangledown$})), respectively.}
 \label{fig:allErrorEstimates}
\end{figure}
\begin{figure}[htb]
 \includegraphics[width=0.19\linewidth]{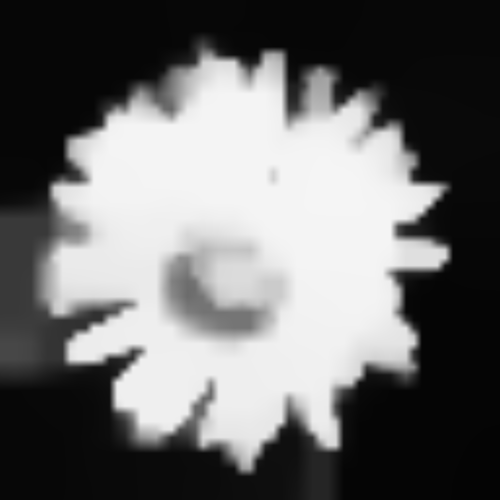}
 \hfill
  \includegraphics[width=0.19\linewidth]{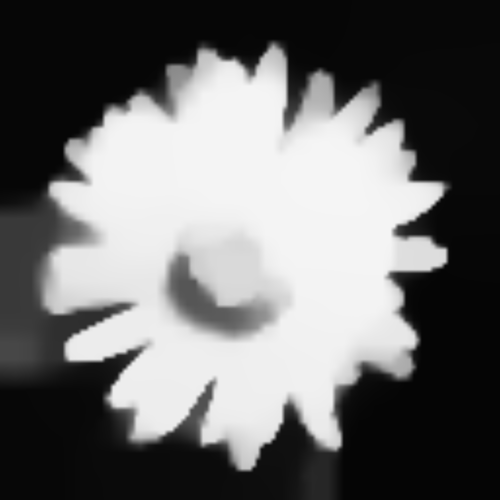}
 \hfill
  \includegraphics[width=0.19\linewidth]{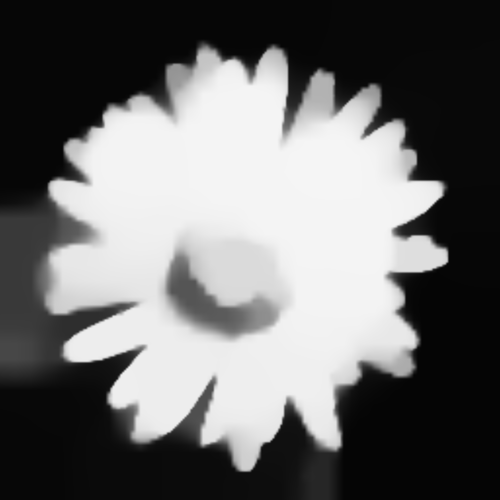}
 \hfill
  \includegraphics[width=0.19\linewidth]{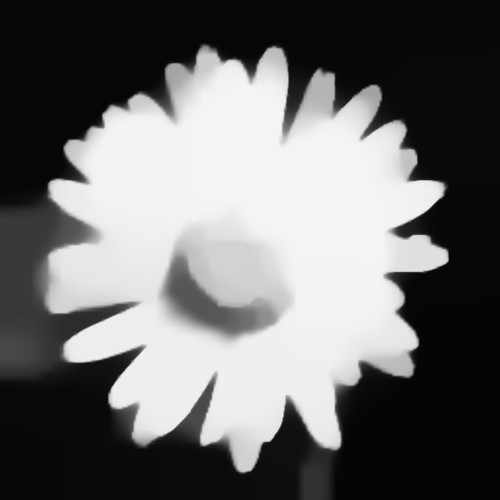}
 \hfill
  \includegraphics[width=0.19\linewidth]{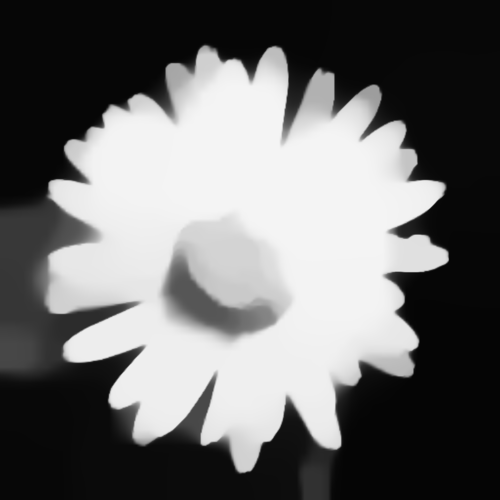}
 \includegraphics[width=0.19\linewidth]{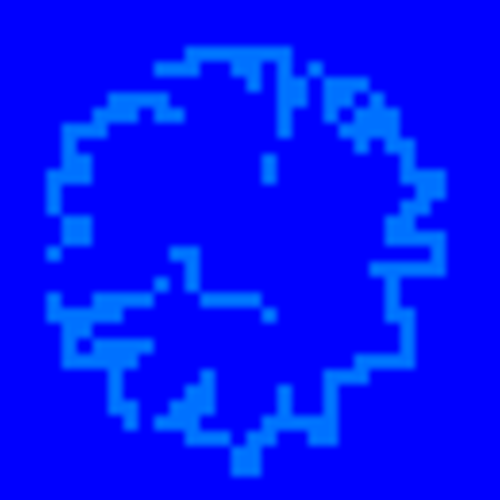}
 \hfill
  \includegraphics[width=0.19\linewidth]{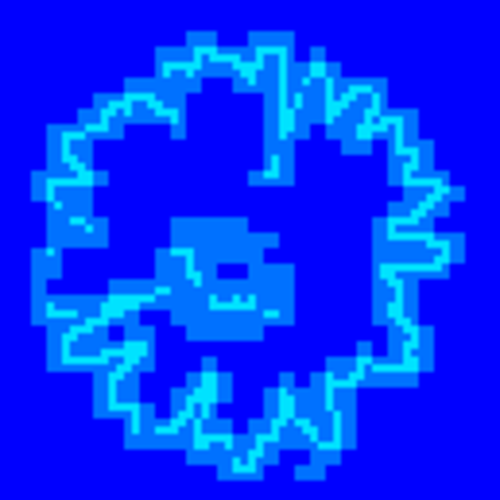}
 \hfill
  \includegraphics[width=0.19\linewidth]{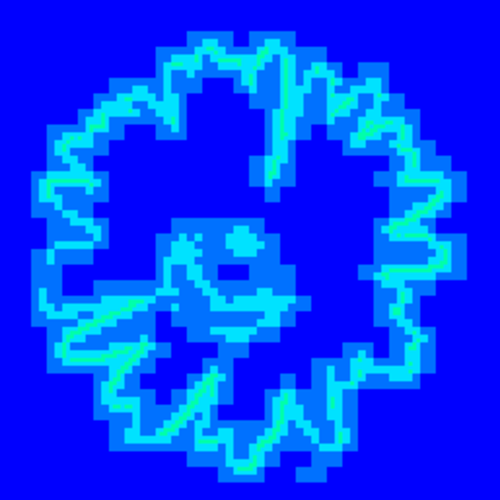}
 \hfill
  \includegraphics[width=0.19\linewidth]{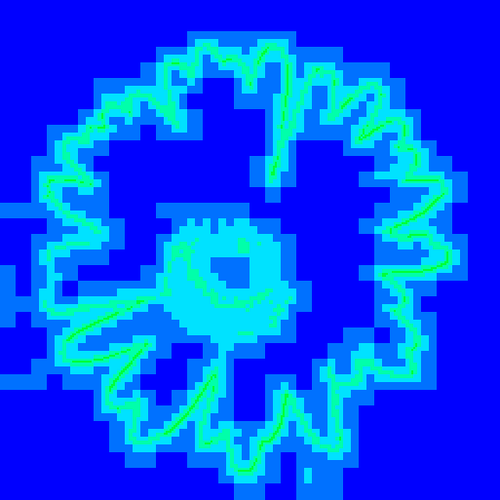}
 \hfill
  \includegraphics[width=0.19\linewidth]{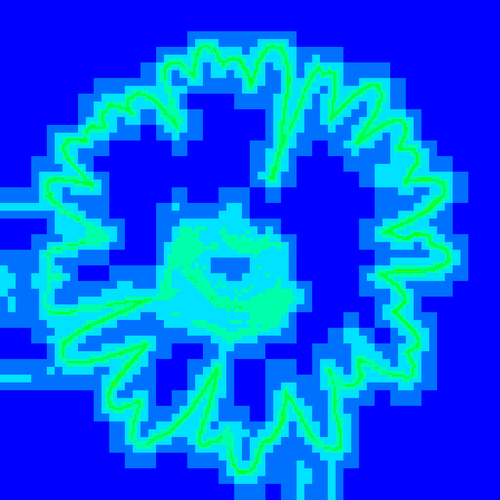}
  \caption{The sequence of solutions $u_h$ and a color coding of the corresponding fineness of the adaptive meshes 
  at the  $1^{st}$,  $2^{nd}$, $3^{rd}$, $4^{th}$ and $5^{th}$ iteration of the adaptive algorithm applied 
  to the input image (c) and computed using the (FE') discretization.}
 \label{fig:meshSequenceFlower}
\end{figure}
Table \ref{table:finalResults} lists (scaled) primal and dual energies, 
$\err_\image^2$, $\eta_{optimal}$ (the $\eta$ value corresponding to  the optimal a posteriori error bound for given $\err^2_u$), and $\err_\chi$ for all input images after the $10$th refinement step of the adaptive algorithm.
The value of $\err_\chi$ peaks for the application (d) due to the relatively low image resolution.

Figure \ref{fig:allErrorEstimates} plots the error estimator $\err_\image^2$
after each refinement step for all input images and both finite element discretizations. In most numerical experiments,
the scheme (FE) performs slightly better than the scheme (FE'). For the flower image, the sequence of adaptive meshes and solutions resulting from the adaptive algorithm for the discretization (FE') is depicted in Figure \ref{fig:meshSequenceFlower}.
Figure \ref{fig:resultsCGPrimalDual} displays solutions for the discretization (FE') and the corresponding adaptive meshes together with color coded deciles of $u_h$, and the graphs of $\eta\mapsto\jumpArea[\Uh,\eta]$ and $\eta\mapsto\err_\chi$.
Note that the displayed deciles explicitly indicate the sets $S_\eta$ for $\eta=0.1,\, 0.2, \, 0.3, \, 0.4$.
Moreover, Figures \ref{fig:resultsPrimalRemaining} and \ref{fig:resultsCameraPrimalRemaining} show the relaxed solution $u_h$ and the thresholded solution $\chi_h$ 
for the remaining discretization schemes and input images.

\setlength{\unitlength}{0.05\textwidth}
\begin{figure}[t]
\resizebox{\linewidth}{!}{
\begin{picture}(20,20)
\put(0,15){\includegraphics[width=0.24\textwidth]{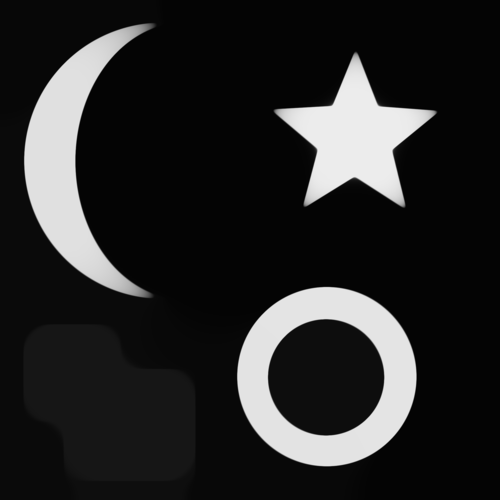}}
\put(5,15){\includegraphics[width=0.24\textwidth]{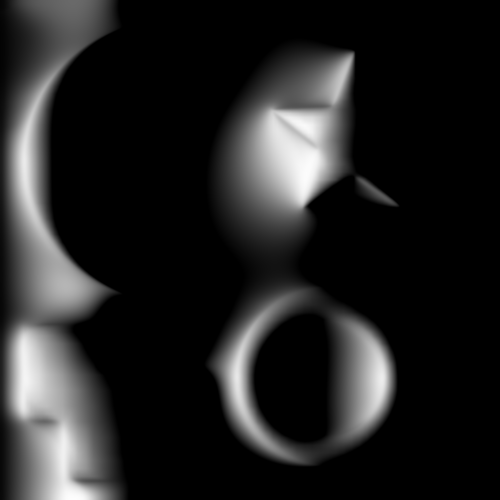}}
\put(10,15){\includegraphics[width=0.24\textwidth]{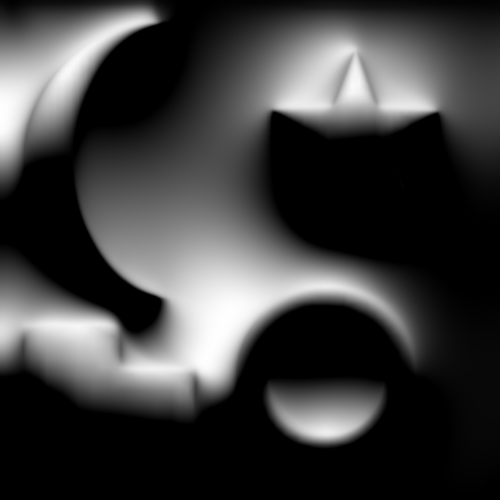}}
\put(15,15){\includegraphics[width=0.24\textwidth]{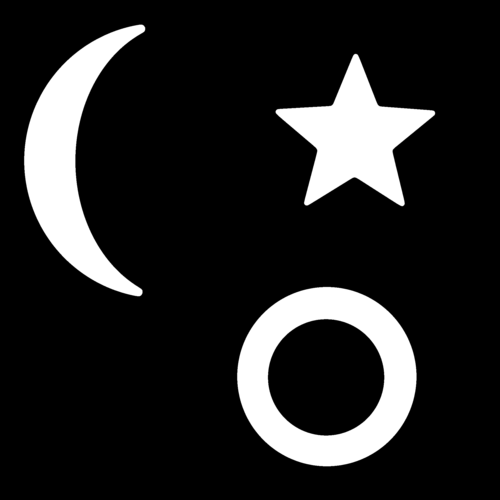}}
\put(0,10){\includegraphics[width=0.24\textwidth]{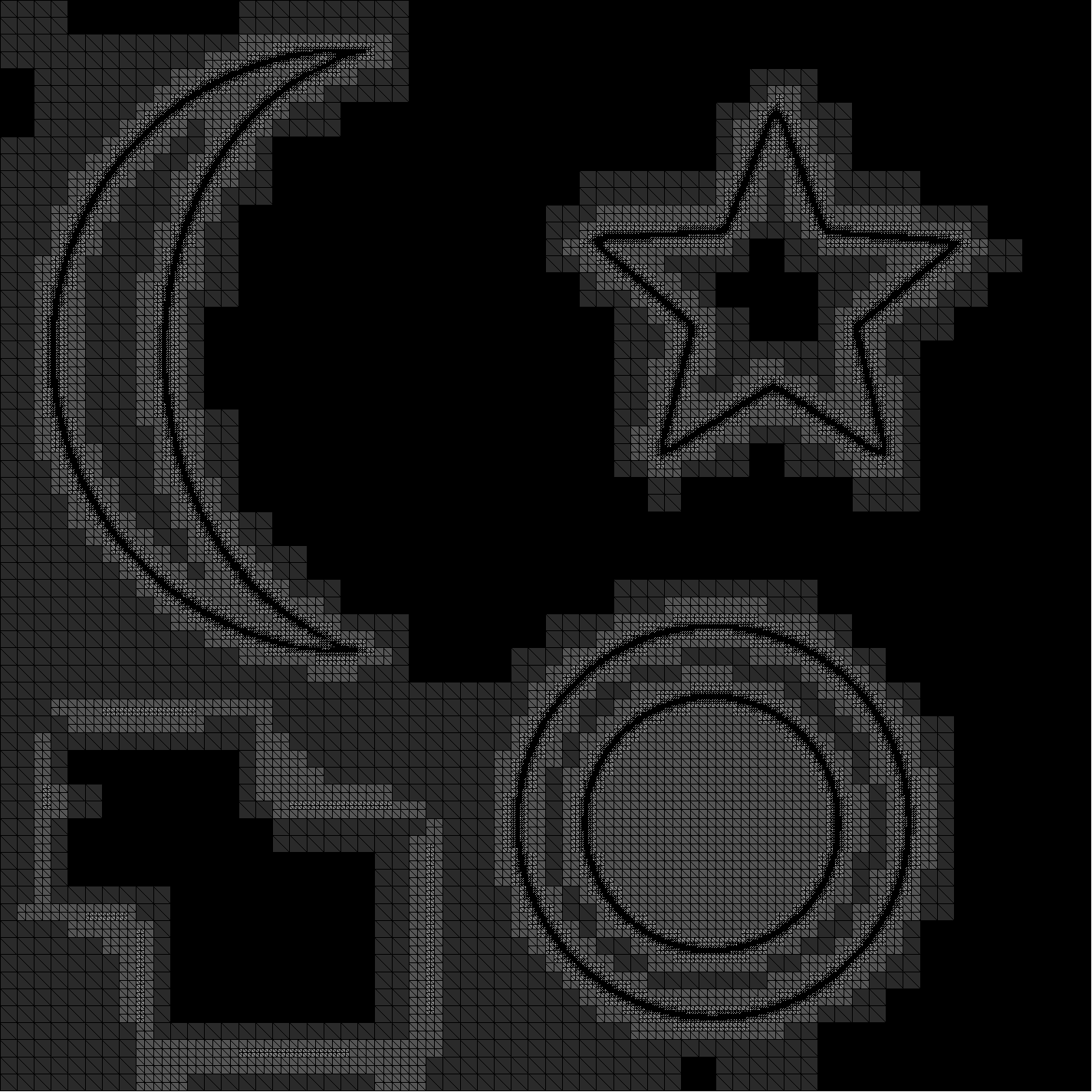}}
\put(5,10){\includegraphics[width=0.24\textwidth]{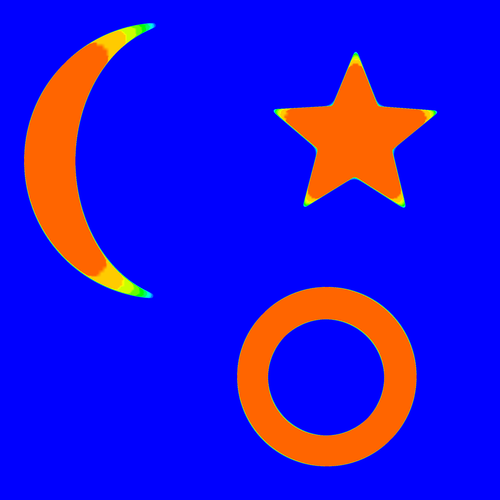}}
\put(10,10){\includegraphics[angle=90, width=0.02\textwidth, height=0.24\textwidth]{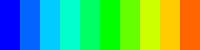}}
\put(10.5,10){$0$}
\put(10.5,14.5){$1$}
\put(10.2,10){\includegraphics[width=0.5\textwidth,height=0.24\textwidth]{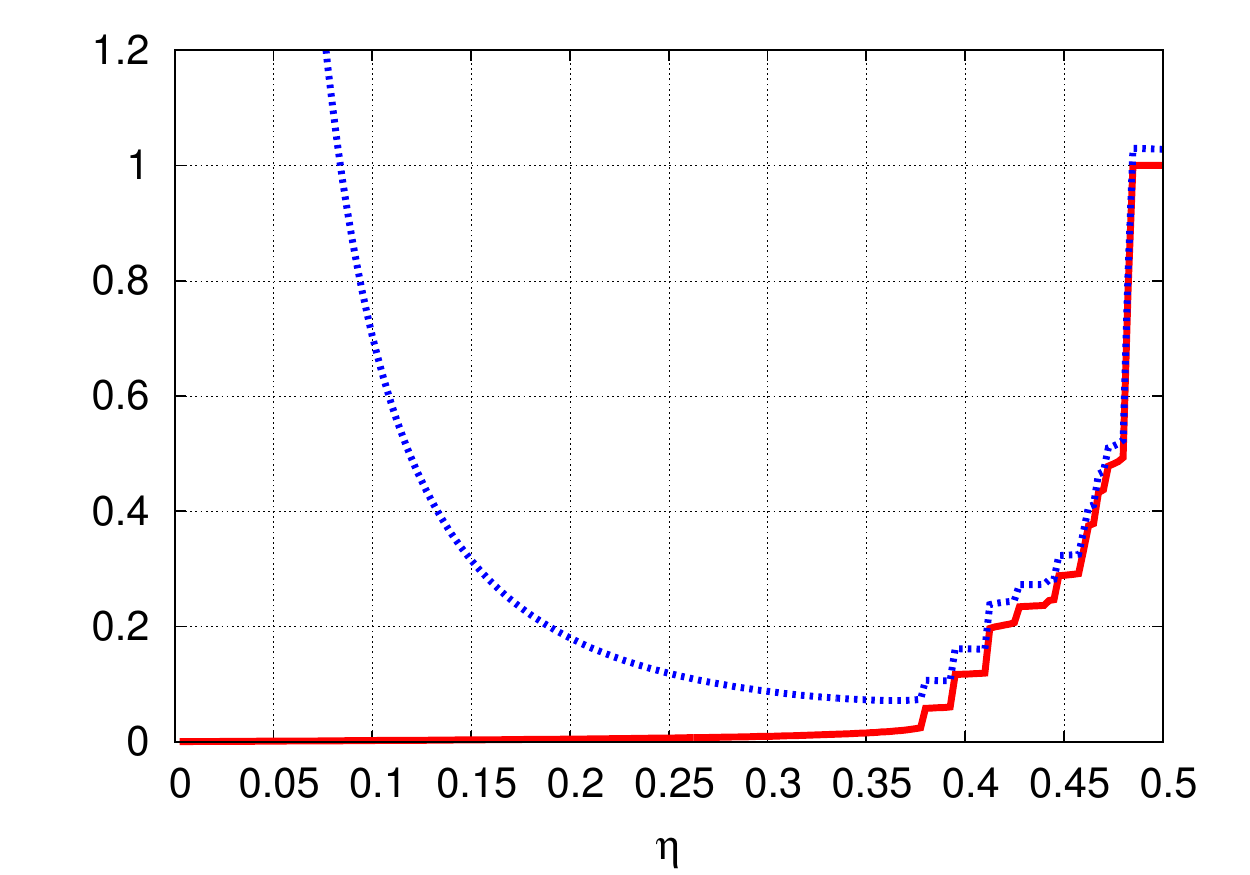}}
\put(0,5){\includegraphics[width=0.24\textwidth]{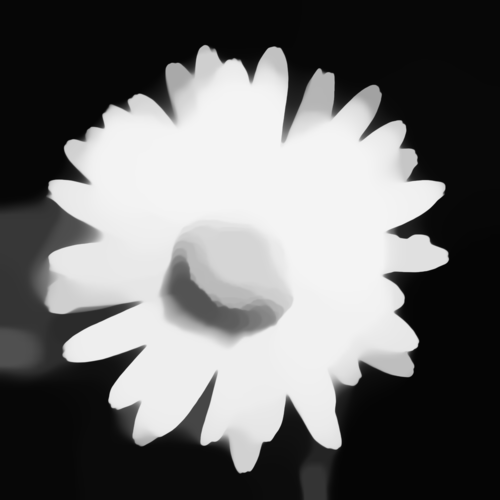}}
\put(5,5){\includegraphics[width=0.24\textwidth]{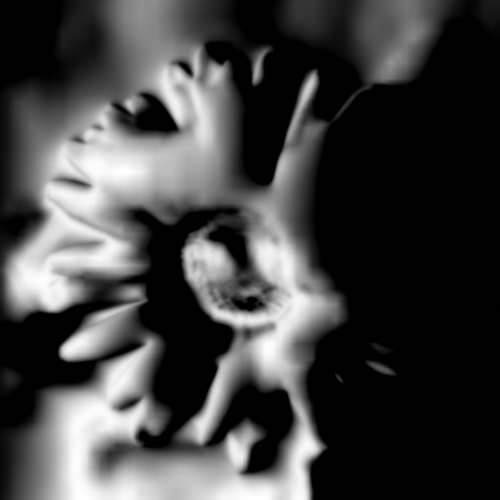}}
\put(10,5){\includegraphics[width=0.24\textwidth]{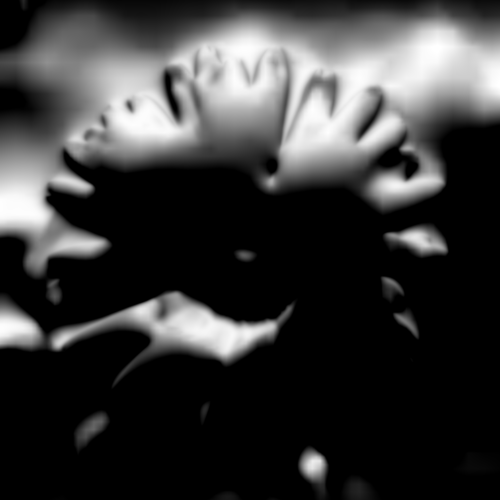}}
\put(15,5){\includegraphics[width=0.24\textwidth]{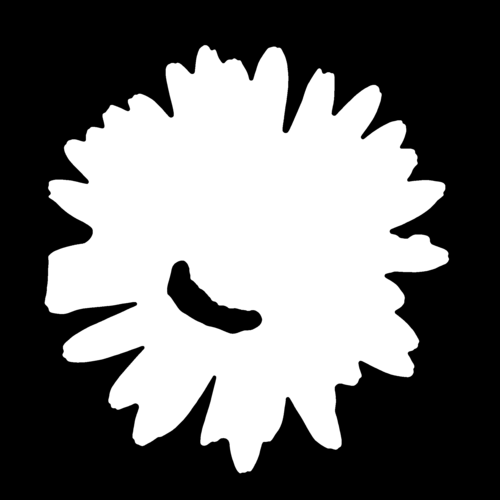}}
\put(0,0){\includegraphics[width=0.24\textwidth]{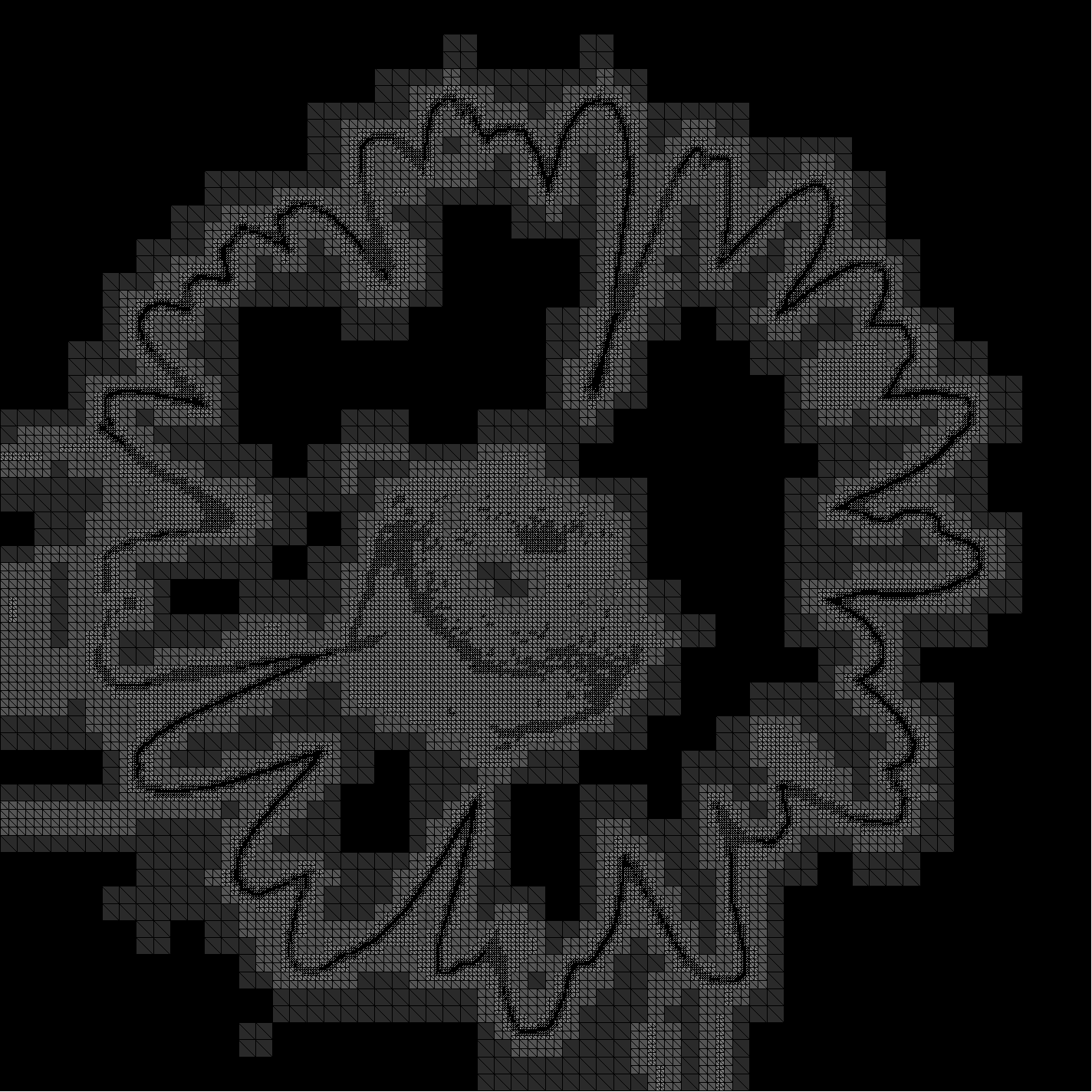}}
\put(5,0){\includegraphics[width=0.24\textwidth]{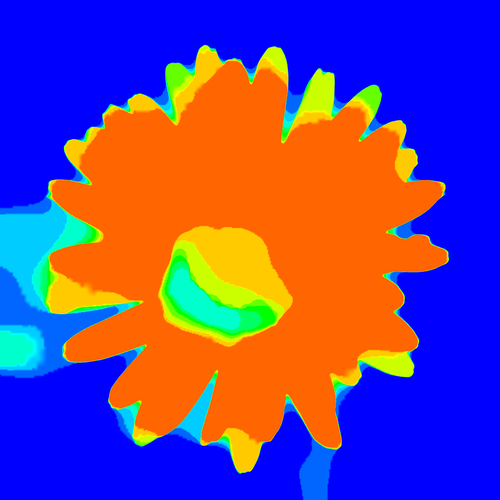}}
\put(10,0){\includegraphics[angle=90, width=0.02\textwidth, height=0.24\textwidth]{images/scale.png}}
\put(10.5,0){$0$}
\put(10.5,4.5){$1$}
\put(10.2,0){\includegraphics[width=0.5\textwidth,height=0.24\textwidth]{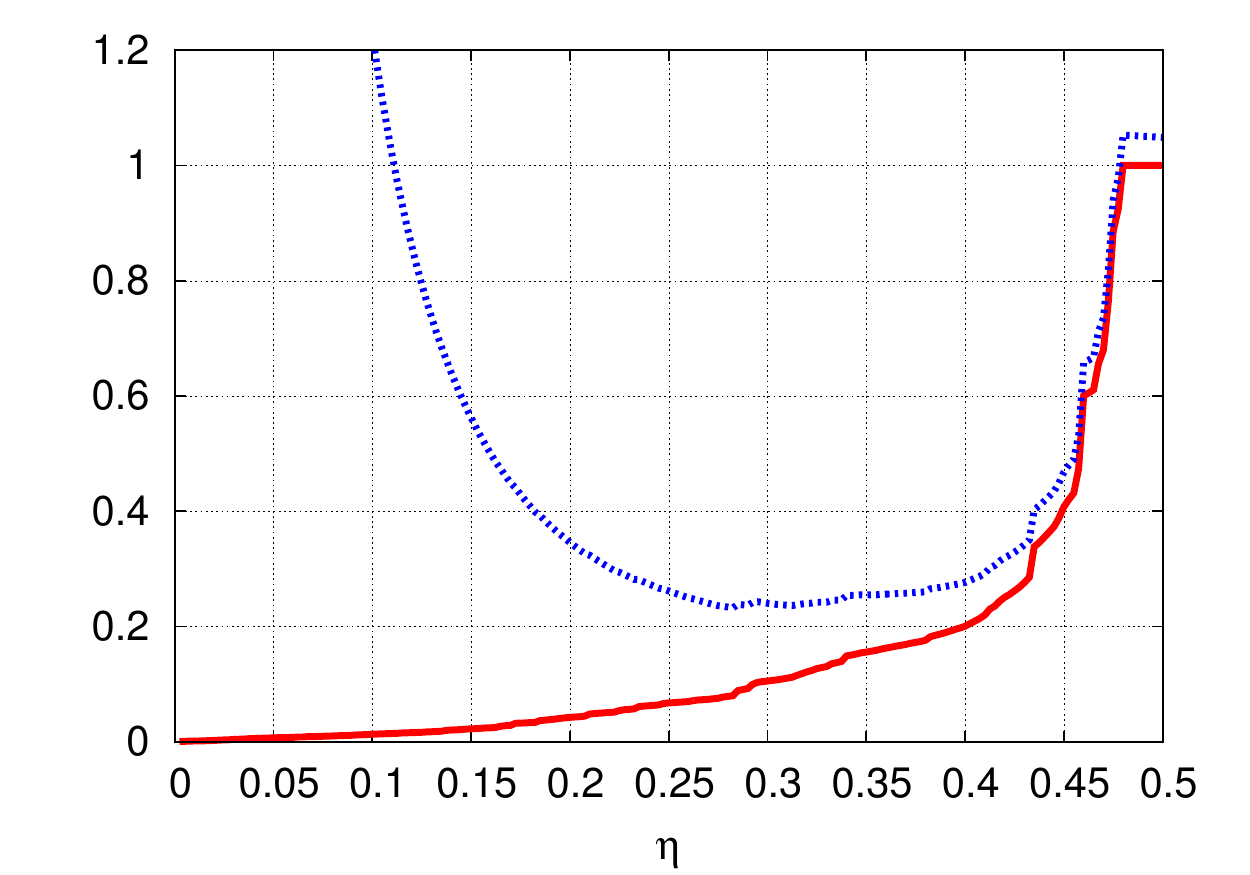}}
\end{picture}
}
\caption{For images (b) (first and second row) and (c) (last two rows) and the discretization (FE') 
the components of the relaxed solution $u_h$, $(p_h)_1$, $(p_h)_2$, the resulting solution $\chi_h$ are shown after the $10^{th}$ iteration of the adaptive scheme. 
In the second and fourth row the adaptive grid (after the $6^{th}$ refinement step), deciles of the discrete solution $u_h$ encoded with different colors, and 
the functions $\eta\mapsto\jumpArea[\Uh,\eta]$ (red solid line)  and $\eta \mapsto \err_\chi$ (blue dashed line) are rendered.}
\label{fig:resultsCGPrimalDual}
\end{figure}
\begin{figure}[htb]
\centering
\resizebox{\linewidth}{!}{
\begin{tikzpicture}
\node[anchor=south west] at (0,2.2) {\includegraphics[width=0.15\linewidth]{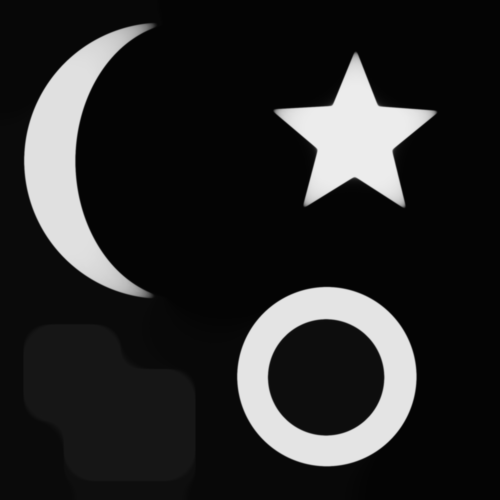}};
\node[anchor=south west] at (2.2,2.2) {\includegraphics[width=0.15\linewidth]{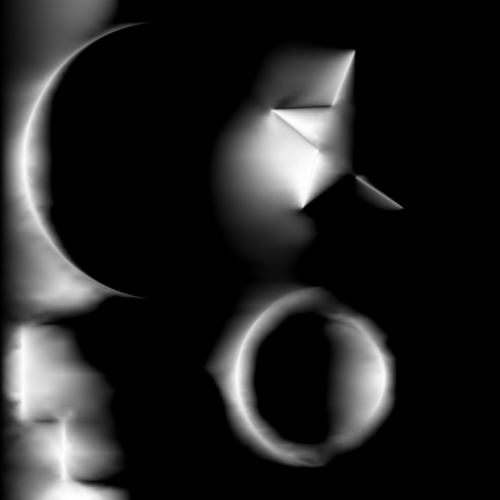}};
\node[anchor=south west] at (4.4,2.2) {\includegraphics[width=0.15\linewidth]{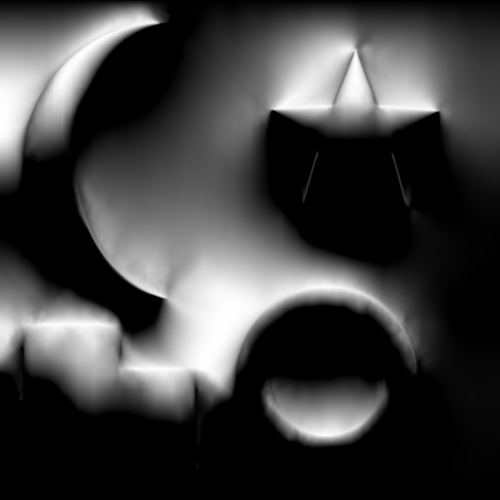}};
\node[anchor=south west] at (6.8,2.2) {\includegraphics[width=0.15\linewidth]{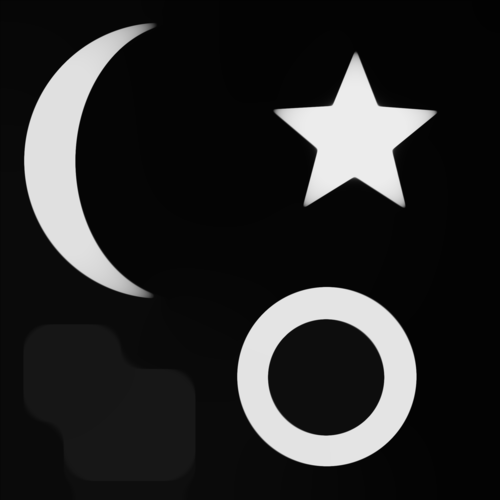}};
\node[anchor=south west] at (9,2.2) {\includegraphics[width=0.15\linewidth]{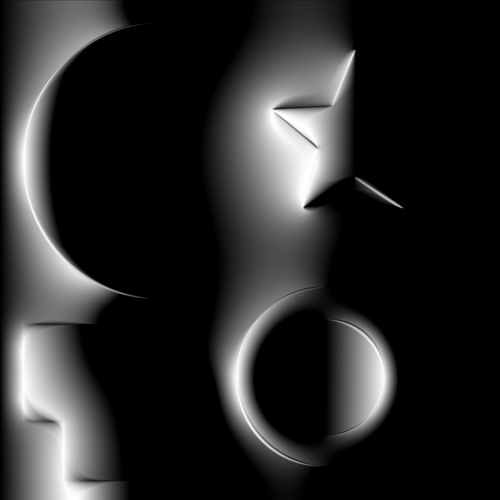}};
\node[anchor=south west] at (11.2,2.2) {\includegraphics[width=0.15\linewidth]{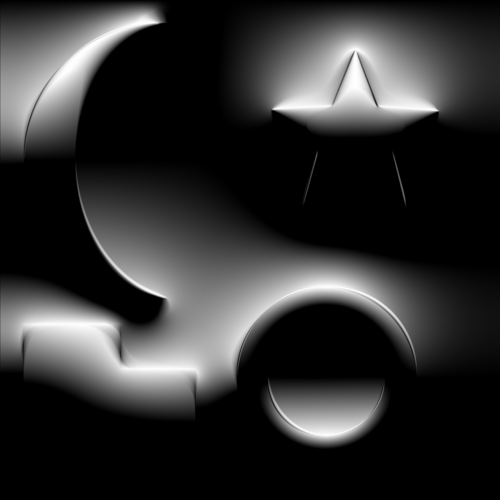}};
\draw [thick] (6.75,0) -- (6.75, 4.5);
\node[anchor=south west] at (0,0) {\includegraphics[width=0.15\linewidth]{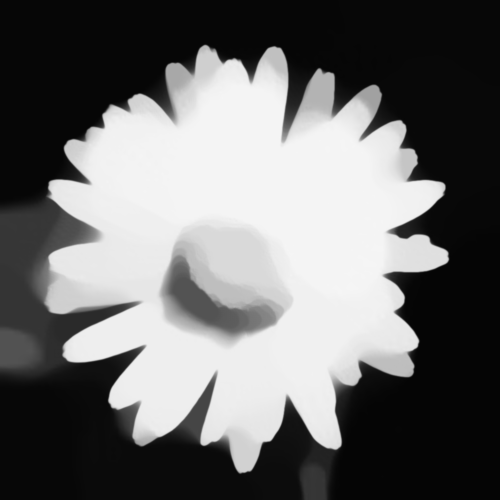}};
\node[anchor=south west] at (2.2,0) {\includegraphics[width=0.15\linewidth]{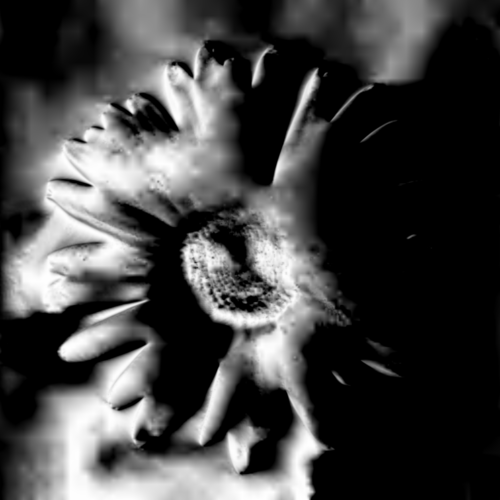}};
\node[anchor=south west] at (4.4,0) {\includegraphics[width=0.15\linewidth]{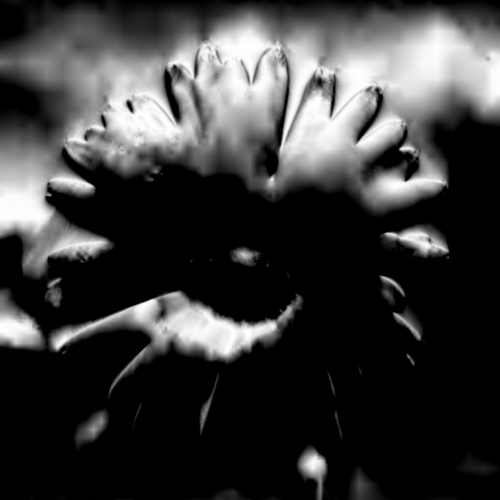}};
\node[anchor=south west] at (6.8,0) {\includegraphics[width=0.15\linewidth]{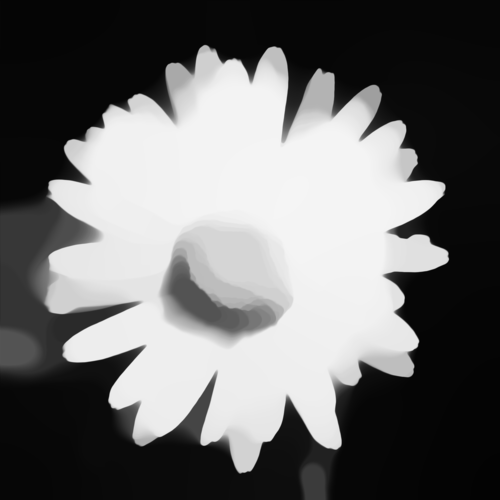}};
\node[anchor=south west] at (9,0) {\includegraphics[width=0.15\linewidth]{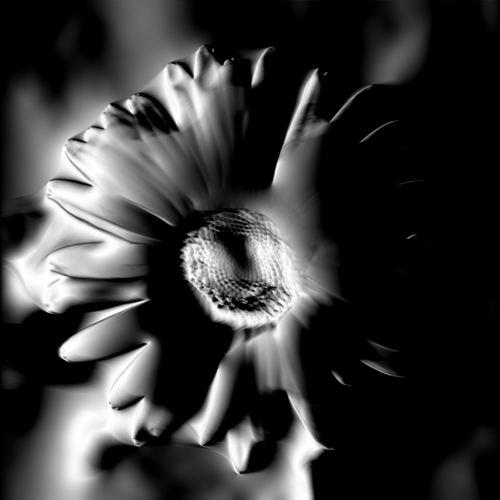}};
\node[anchor=south west] at (11.2,0) {\includegraphics[width=0.15\linewidth]{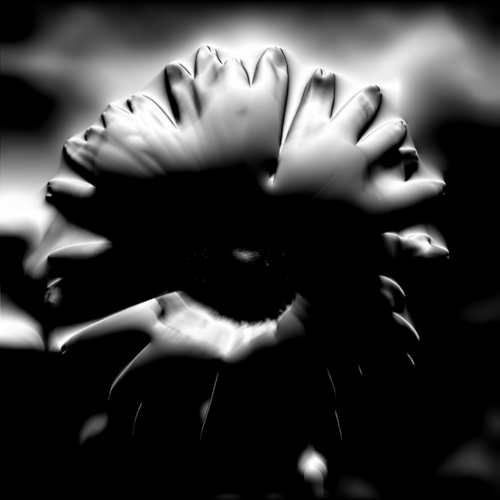}};
\end{tikzpicture}
}
\caption{Relaxed solution $u_h$, $(p_h)_1$, $(p_h)_2$ for the input image (b) (top row) and (c) (bottom row) using the discretization 
(FE) (left, after $10$ iterations of the adaptive algorithm) and (FD) (right).}
\label{fig:resultsPrimalRemaining}
\end{figure}
\begin{figure}[htb]
\includegraphics[width=0.16\linewidth]{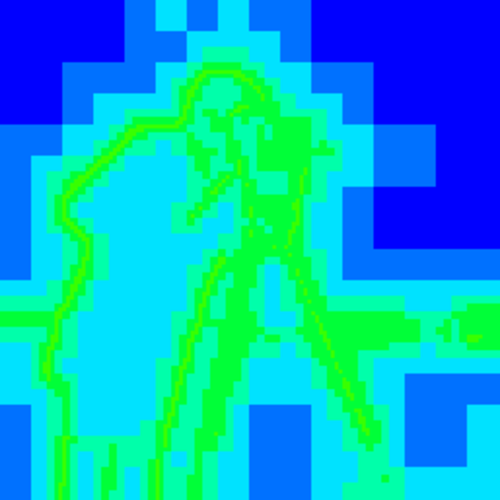}
\hfill
\includegraphics[width=0.16\linewidth]{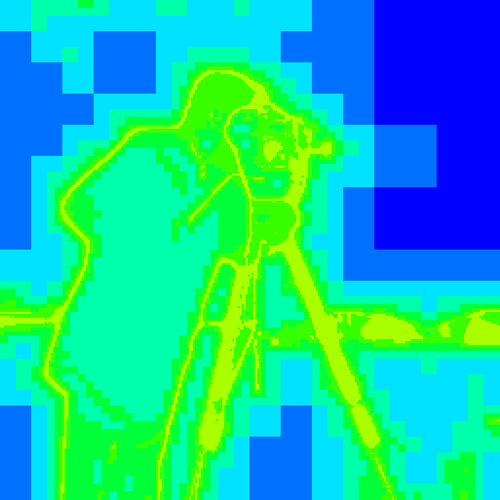}
\hfill
\includegraphics[width=0.16\linewidth]{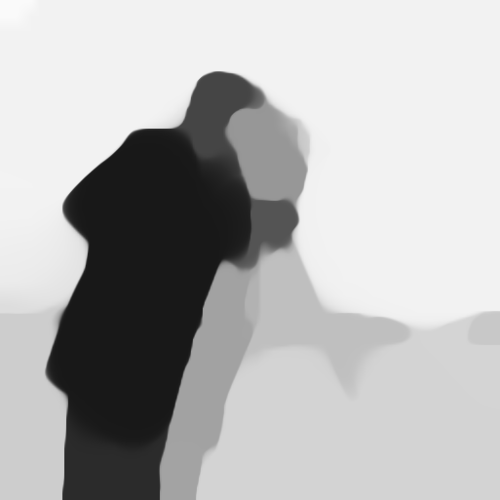}
\hfill
\includegraphics[width=0.16\linewidth]{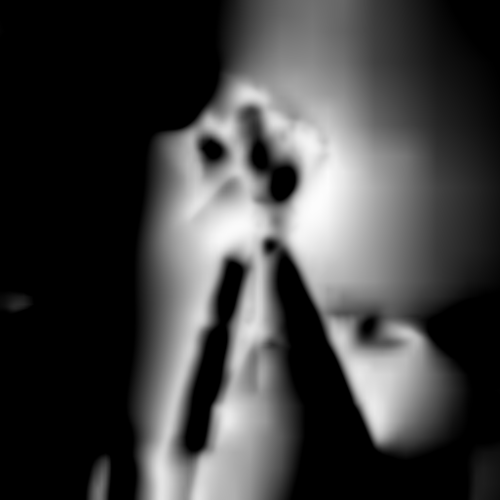}
\hfill
\includegraphics[width=0.16\linewidth]{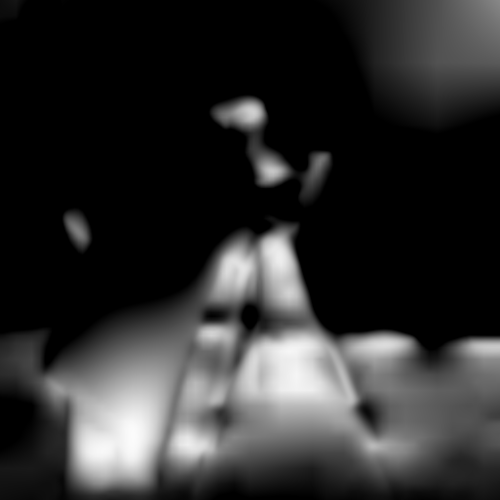}
\hfill
\includegraphics[width=0.16\linewidth]{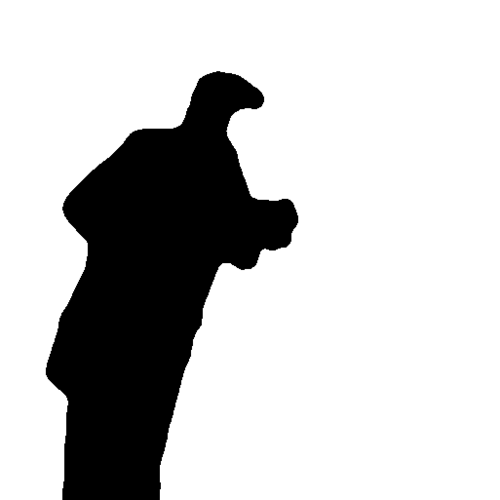}
\caption{The mesh in the $5^{th}$ and $10^{th}$ iteration, the relaxed solution $u_h$, $(p_h)_1$, $(p_h)_2$ and $\chi_h$ for the input image (d) using the discretization (FE')  after $10$ iterations of the algorithm.}
\label{fig:resultsCameraPrimalRemaining}
\end{figure}

Finally, we applied the above methods to an analytic function consisting
of a weight\-ed sum of two Gaussian kernels. To this end, in each step the functionals and the error estimator 
are evaluated on the current adaptive grid and not on a prefixed full resolution grid.
The results are shown in Figure \ref{fig:gaussian} (with parameters $c_1=0.495349$, $c_2=0.056845$  and $\nu=5\cdot 10^{-3}$).

\setlength{\unitlength}{0.05\textwidth}
\begin{figure}[t]
\resizebox{\linewidth}{!}{
\begin{picture}(20,9)
\put(0,5){\includegraphics[width=0.24\textwidth]{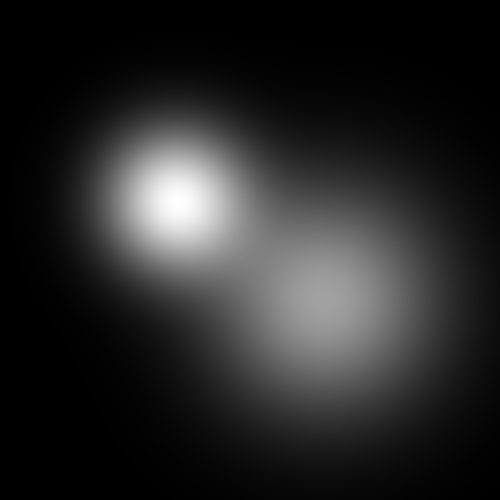}}
\put(5,5){\includegraphics[width=0.24\textwidth]{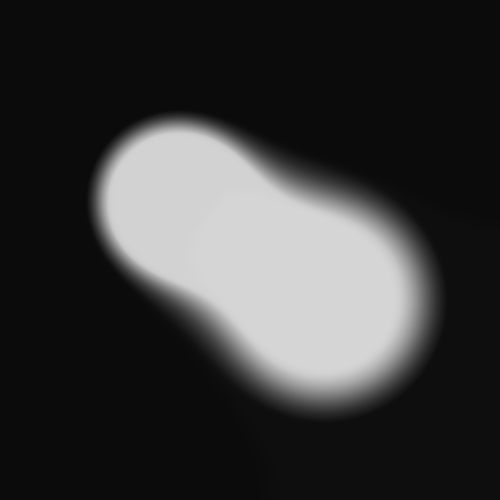}}
\put(10,5){\includegraphics[width=0.24\textwidth]{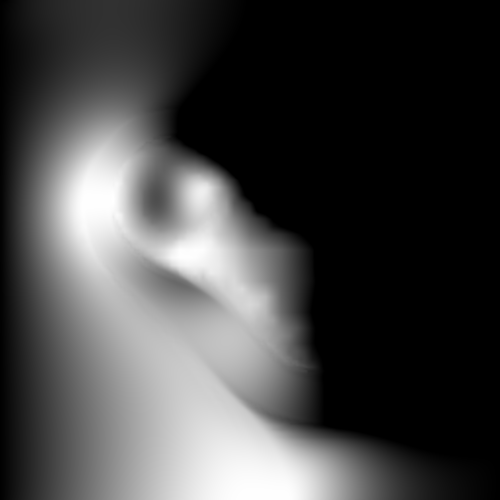}}
\put(15,5){\includegraphics[width=0.24\textwidth]{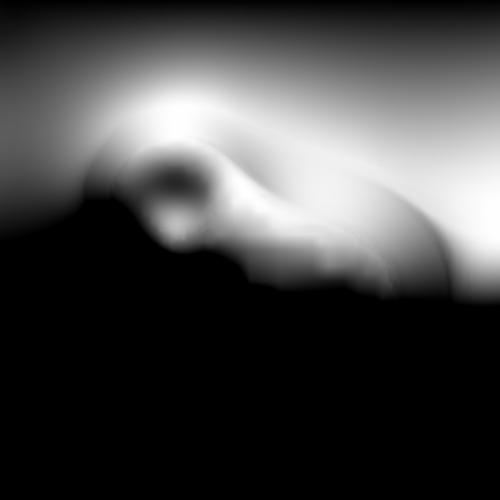}}
\put(0,0){\includegraphics[width=0.24\textwidth]{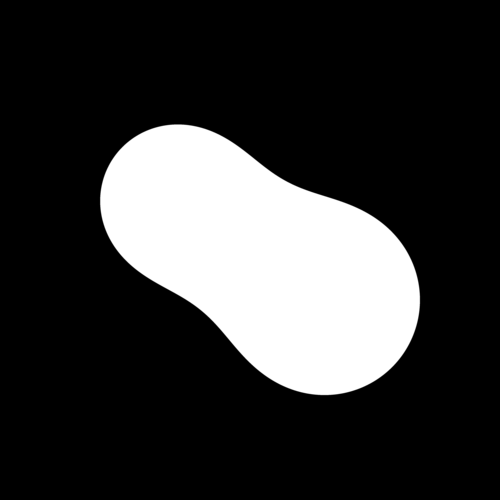}}
\put(5,0){\includegraphics[width=0.24\textwidth]{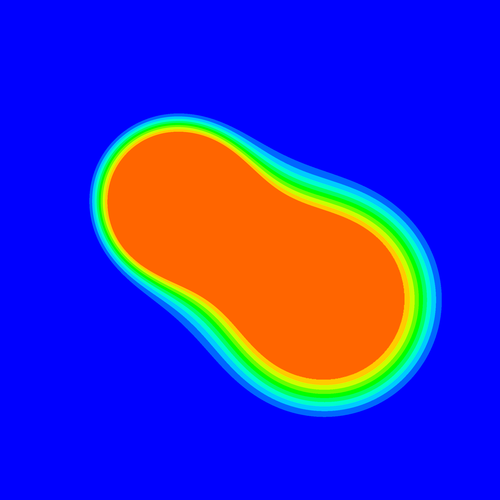}}
\put(10,0){\includegraphics[angle=90, width=0.02\textwidth, height=0.24\textwidth]{images/scale.png}}
\put(10.5,0){$0$}
\put(10.5,4.5){$1$}
\put(10.5,-0.2){\includegraphics[width=0.48\textwidth,height=0.26\textwidth]{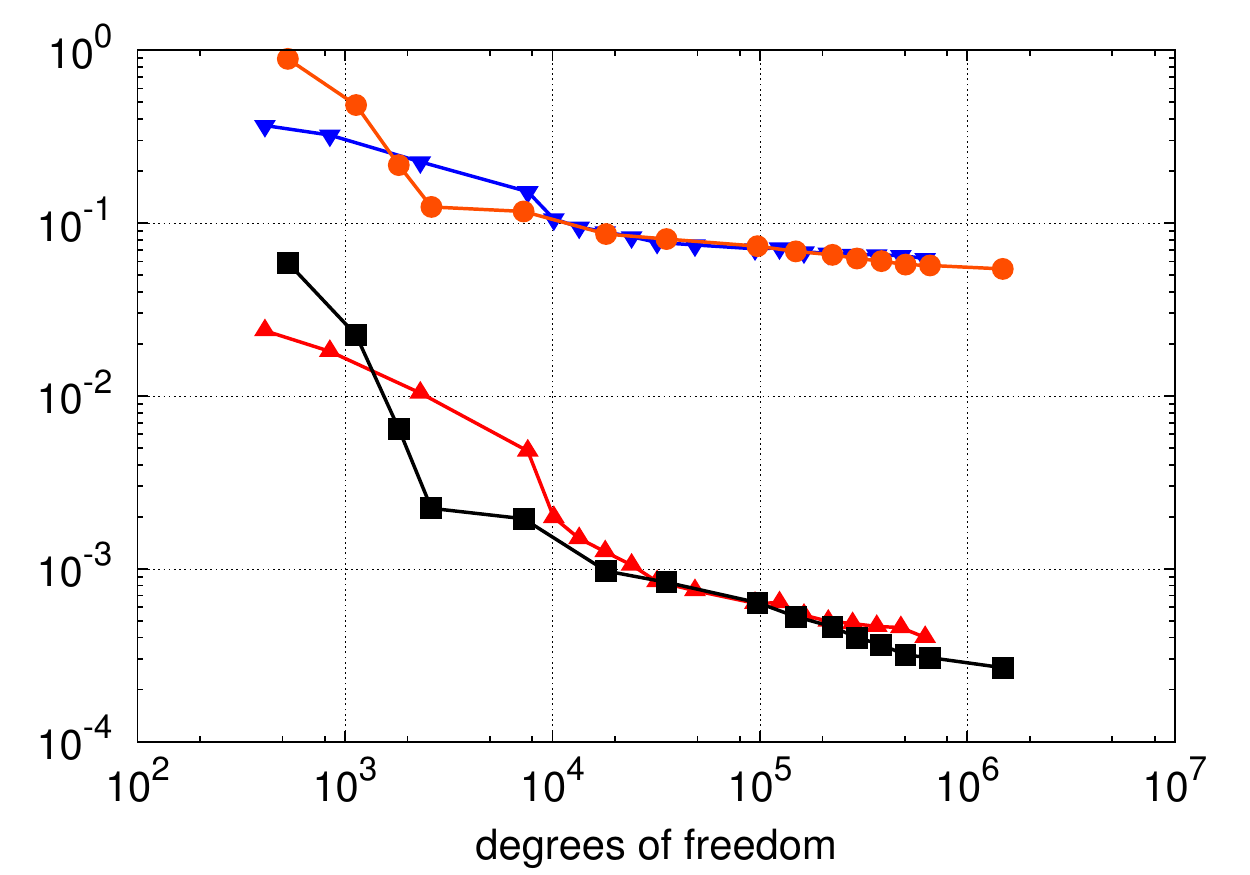}}
\end{picture}
}
\caption{First row: Input image $u_0$ composed by the superposition of two Gaussian kernels, numerical solutions $\image_h$, $(p_h)_1$ and $(p_h)_2$ computed via the adaptive algorithm using discretization (FE'). Second row: $\chi_h$, deciles and the error estimators
in a log-log plot ($\err_\image^2$ in red {\color{red} $\filledmedtriangleup$}, $\err_\chi$ in blue {\color{blue} $ \filledmedtriangledown$}  for the discretization (FE') and $\err_\image^2$ in black {\color{black} $\blacksquare$}, $\err_\chi$ in brown {\color{gnuplotbrown} $ \bullet$} for the discretization (FE)).}
\label{fig:gaussian}
\end{figure}

\section{Conclusions}
We have investigated the a posteriori error estimation for the binary Mumford-Shah model and applied this estimate to two 
different adaptive finite element discretizations in comparison to a nonadaptive finite difference scheme on a regular grid.
The  proposed finite element discretizations in combination with the adaptive meshing strategy lead to a substantial reduction of the required degrees of freedom 
with error values  $\err_\image^2$ and $\err_\chi$ of about the same magnitude as for a standard finite difference scheme.
To improve the resulting estimate of the duality gap $\ERel[v]+\DRel[q]$, the finite element schemes (FE) and (FE') require some oscillation damping smoothing  
in a post-processing step.

The proposed approach to a posteriori estimates for the binary Mumford-Shah model derived in this paper can be applied straightforwardly to more general problems in computer vision.
In fact, the calibration method developed by Alberti, Bouchitt\'e and Dal Maso \cite{AlBo03} provides a convex relaxation of non-convex functionals of 
Mumford-Shah type via the lifting of a variational problem on a $n$-dimensional domain to a minimization problem over characteristic functions of subgraphs in $n+1$ dimensions. 
In the context of non-convex functionals in vision, this approach was studied by Pock \etal  \cite{PoCrBi09,PoCrBi10}. Applications of such functionals include the computation of minimal partitions \cite{ChCrPo08, PoScGr08}, the depth map identification from stereo images or the robust extraction of optimal flow fields \cite{PoCrBi10}.
Here, an adaptive mesh strategy is expected to have an even larger pay-off due to the increased dimension.

\section*{Acknowledgements}
A. Effland and M. Rumpf acknowledge support of the Collaborative Research Centre 1060 and the Hausdorff Center for Mathematics,  both funded by the German Science foundation. 
B. Berkels was funded in part by the Excellence Initiative of the German Federal and State Governments. 

{\small
\bibliographystyle{siam}
\bibliography{Bibtex/all,Bibtex/library,Bibtex/own}
}

\end{document}